\documentclass[12pt]{article} 
\usepackage{amsmath}
\usepackage{amssymb}
\usepackage{amsthm} 
\usepackage{mathrsfs}  
 
\usepackage[margin=2.0cm]{geometry} 

\usepackage{hyperref}
\usepackage{multirow} 
\usepackage{color} 
\usepackage{enumerate}
\usepackage{tikz}

\newtheorem{lemma}{Lemma}[section]%
\newtheorem{theorem}{Theorem}%
\newtheorem{proposition}[lemma]{Proposition}%
\newtheorem{hypothesis}[lemma]{Hypothesis}%
\newtheorem{corollary}[lemma]{Corollary}%
\numberwithin{equation}{section}

\begin{document}
\title{Exceptional groups and the $s$-arc-transitivity of vertex-primitive digraphs, I}

\author{ \\ Fu-Gang Yin$^{\rm a,}$\footnotemark, Lei Chen$^{\rm b}$ \\ 
{\small\em $^{\rm a}$ 
 School of Mathematics and Statistics, Beijing Jiaotong University, Beijing, P.R. China}\\ 
{\small\em $^{\rm b}$ Faculty of Mathematics, Bielefeld University, Bielefeld, Germany}
}

\renewcommand{\thefootnote}{\fnsymbol{footnote}}  
\footnotetext[1]{Corresponding author.  \\
E-mails:   fgyin@bjtu.edu.cn(Fu-Gang Yin), lchen@math.uni-bielefeld.de(Lei Chen)}

\date{} 
\maketitle 

\begin{abstract}

In this paper, we study the primitive actions of almost simple exceptional groups of Lie type on \(s\)-arc-transitive digraphs. Our motivation is the following question posed by Giudici and Xia: Is there an upper bound on $s$ for finite vertex-primitive $s$-arc-transitive digraphs that are not directed cycles? In a 2018 paper,
Giudici and Xia reduced this question to the case where the automorphism group of the digraph is an almost simple group with socle \(L\).
Subsequently, it has been proved that $s\leq 2$ when \(L\) is a linear, symplectic or alternating group, and $s\leq 1$ when \(L\) is a Suzuki group, a small Ree group, or one of  $22$  specific sporadic groups.
In this paper, we prove that $s\leq 2$ when \(L\) is  $ {}^3D_4(q)$, $G_2(q)$ (including $G_2(2)'$), ${}^2F_4(q)$ (including ${}^2F_4(2)'$), $F_4(q)$, $E_6(q)$ or ${}^2E_6(q)$. 

\end{abstract}
\section{Introduction}

Given the completion of the proof of Classification of Finite Simple Groups (CFSG), which divides all the non-abelian finite simple groups into classical groups, exceptional groups of Lie type, alternating groups and sporadic groups, it is natural to study the subgroup structure of simple groups and there is a vast literature on this topic, finding a wide range of applications. For example, the point stabilisers of a primitive permutation group are maximal subgroups, so knowledge of maximal subgroups plays an important role in the study of primitive group actions.

The general structure of the maximal subgroups of almost simple groups of classical and alternating types is described by important theorems of Aschbacher \cite{Aschbacher1984} and O'Nan-Scott \cite{LPS1987}, respectively.
The maximal subgroups of the almost simple groups of exceptional Lie type have been studied by various authors. For the groups with socle \(G_{2}(q)\), \({}^{2}B_{2}(q)\), \({}^{2}G_{2}(q)\), \({}^{3}D_{4}(q)\) and \({}^{2}F_{4}(q)\), a complete classification of the maximal subgroups (up to conjugacy) is given in \cite{Cooperstein1980}, \cite{Kleidman19882},\cite{Kleidman1988} and \cite{Malle1991}. And recent work of Craven \cite{C2023} has extended this to the groups with socle \(F_{4}(q)\), \(E_{6}(q)\) and \({}^{2}E_{6}(q)\).
It remains an open problem to complete the analysis for the groups with socle \(E_7(q)\) and \(E_8(q)\).

Let \(H\) be an almost simple group such that 
\[ \mathrm{Soc}(H)\in\{{}^3D_4(q), G_2(q), F_{4}(q), {}^{2}F_{4}(q), E_{6}(q), {}^{2}E_{6}(q)\}. 
\]
In this paper, we will study factorisations and orbitals of maximal subgroups of \(H\), and we will adopt a computational approach to study the structure of maximal rank subgroups of \(H\). This will allow us to make progress on a longstanding problem in algebraic combinatorics concerning the vertex-primitive \(s\)-arc-transitive digraphs, which dates back to a paper \cite{P1989} of Praeger from 1989.

A digraph $\mathit{\Gamma}$ is defined as a pair $(V, \to)$, where $V$ is a set of vertices and $\to$ is an antisymmetric irreflexive binary relation on $V$. 
The vertex set of $\mathit{\Gamma}$ is also denoted by  $V(\mathit{\Gamma})$. 
For a positive integer $s$, an \emph{$s$-arc} of $\mathit{\Gamma}$ is a sequence of $s+1$ vertices $ v_0,v_1, \dots, v_s $ such that $v_i\rightarrow v_{i+1}$
for each $i\in\{0,1, \dots, s-1\}$. 
For a subgroup $H$ of the full automorphism group $\mathrm{Aut}(\mathit{\Gamma})$  of $\mathit{\Gamma}$, we say that $\mathit{\Gamma}$ is \emph{$(H,s)$-arc-transitive} if $H$ acts transitively on the $s$-arc set.
Furthermore,  $\mathit{\Gamma}$ is said to be \emph{$H$-vertex-transitive} if $H$ is transitive on  $V(\mathit{\Gamma})$, and \emph{$H$-vertex-primitive} if  $H$ is primitive on  $V(\mathit{\Gamma})$.

For an $H$-vertex-transitive digraph $\mathit{\Gamma}$, the in-neighbours and out-neighbours of each vertex have the same size, which is called the \emph{valency} of $\mathit{\Gamma}$, and it is straightforward to prove that if  $\mathit{\Gamma}$ is $(H,s+1)$-arc-transitive then it must also be $(H,s)$-arc-transitive. The graphs and digraphs discussed in this paper are all connected, vertex-transitive and are not directed cycles.
In 1989, Praeger~\cite{P1989} constructed an infinite family of $s$-arc-transitive, but not $(s+1)$-arc-transitive, digraphs of valency $v$, for each pair of positive integers $s$ and $v$. This is in stark contrast to the situation of graphs, for which a celebrated theorem of Weiss~\cite{Weiss1981} states that \(s\leq7\).

Since then, several new constructions of $s$-arc-transitive digraphs have been introduced (see~\cite{1995CLP} for references).  
However, no vertex-primitive $s$-arc-transitive digraph with $s\geq 2$ was discovered until 2017, when Giudici, Li and Xia~\cite{GLX2017} constructed an infinite family of vertex-primitive $2$-arc-transitive digraphs. 
In \cite{YFX2023}, it is proved that the smallest vertex-primitive $2$-arc-transitive digraph comes from the construction in \cite{GLX2017}, and it has $30,758,154,560$ vertices.
To date, no vertex-primitive $3$-arc-transitive examples have been found.  In 2018, Giudici and Xia~\cite{GX2017} posed the following question:

\medskip
\noindent{\bf Question.}
Is there an upper bound on $s$ for vertex-primitive $s$-arc-transitive digraphs that are not directed cycles? 

\medskip

A systematic investigation of the O'Nan-Scott types of primitive groups has reduced the above question to almost simple groups, see [20, Corollary 1.6]. 
More precisely, there exists an absolute constant $C$ such that every vertex-primitive $s$-arc-transitive digraph that is not a directed cycle satisfies the bound $s\leq C$ if and only if the bound holds  for every $(H,s)$-arc-transitive digraph with $H$ a primitive almost simple group.
Consequently, it suffices to establish the existence of, and then determine the smallest possible value of, such a bound $C$ on $s$ within the almost simple case.

Let \(H\) be the automorphism group of a vertex-primitive $s$-arc-transitive digraph, and assume that $H$ is an almost simple group with socle $L$. 
It has been proved that $s\leq 2$ when the simple group $L$ is a linear group~\cite{GLX2019}, a symplectic group~\cite{CGP2024}, or an alternating group~\cite{PWY2020,CCGLPX2024+}. 
In addition, we know that $s\leq 1$ when \(L\) is a Suzuki group ${}^2B_2(q)$ or a small Ree group ${}^2G_2(q)$~\cite{CGP2023}, or a sporadic simple group not isomorphic to $\mathrm{Co}_1$, $\mathrm{Fi}_{24}'$,  $\mathbb{B}$ and $\mathbb{M}$~\cite{YFX2023} (for these four sporadic groups, an upper bound on $s$ remains unknown). 
In this paper, we determine an upper bound on $s$ when \(L\) is one of the following simple exceptional groups of Lie type: $ {}^3D_4(q)$, $G_2(q)$ (including $G_2(2)'$), ${}^2F_4(q)$ (including ${}^2F_4(2)'$), $F_4(q)$, $E_6(q)$ and ${}^2E_6(q)$. Our main result is the following. 

\begin{theorem}\label{th:main}
Let $\mathit{\Gamma}$ be a connected $H$-vertex-primitive $(H, s)$-arc-transitive digraph, where \(H\) is an almost simple group with socle 
\[L\in \{{}^3D_4(q), G_2(q),  {}^2F_4(q),F_4(q),E_6(q), {}^2E_6(q),G_2(2)',{}^2F_4(2)'\} .\] 
Then $s\leq 2$. 
\end{theorem}

Although Theorem~\ref{th:main} shows that $s \leq 2$ in each case, we are not aware of a single example with $s = 2$. 

The results in this paper for the groups with socle \(L={}^3D_4(q)\), \(G_2(q)\), or ${}^2F_4(q)$ formed part of the second author's Ph.D. thesis~\cite{Chen2023}. The remaining exceptional groups with socle $E_7(q)$ and $E_8(q)$ will be addressed in future work.

\section{Notation and preliminaries}
 
Our notation for groups is standard. For a finite group $G$, we denote by $\mathrm{Rad}(G)$ the largest solvable normal subgroup of $G$, by $G^{(\infty)}$ the smallest normal subgroup of $G$ such that $G/G^{(\infty)}$ is solvable, by $\mathbf{O}_p(G)$ the largest normal $p$-subgroup of $G$, and by $\mathrm{Z}(G)$ the centre of $G$.
The socle of $G$ is the product of all its minimal normal subgroups, and $G$ is almost simple if and only if the socle of $G$ is a nonabelian simple group. 

Let $G$ be a permutation group on a finite set.
The stabiliser of a point $\alpha$ in $G$ is denoted by $G_\alpha$, and  the orbit of $G$ containing $\alpha$ is denoted by $\alpha^G$.

Given integers $q\geq 2$ and  $n\geq 2$, a prime number $r$ is called a \emph{primitive prime divisor} of the pair $(q, n)$ if $r$ divides $q^n - 1$ but does not divide $q^i -1 $ for any positive integer $i <n$. Let $\mathrm{ppd}(q, n)$ be the set of primitive prime divisors of $(q, n)$.
The following result is due to Zsigmondy (see for example~\cite[Theorem IX.8.3]{BH1982}).

\begin{proposition}\label{prop:ppd}
Let $q\geq 2$ be a prime power and $n\geq 2$.
Then there exists a prime $r\in \mathrm{ppd}(q, n)$ unless $(q,n)=(2, 6)$, or if $q$ is a Mersenne prime and $n =2$. Furthermore, $r\equiv 1\pmod{n}$.
\end{proposition}

For a positive integer $n$ and prime number $p$, let  $n_p$ be the largest power of $p$ dividing $n$.

\begin{lemma}\label{lm:pffp}
Let $f$ be a positive integer and let $p$ be a prime.
Then $p^f\geq (f_p)^p.$ 
\end{lemma}

\begin{proof}
Suppose that $f_{p}=p^a$ for some positive integer $a$.   
If $a =1$, then $f_{p}=p^a=p $ and so  $p^f \geq p^{f_p}= (f_p)^p$.
If $a\geq 2$, then 
\[ p^a=(p-1+1)^a =\sum_{i=0}^{a}\binom{a}{i}(p-1)^a \geq  \binom{a}{1}(p-1)+\sum_{i\neq 1} 1= a(p-1)+a =ap,\]
 and so $p^f\geq p^{p^a}\geq p^{ap}=(p^a)^p=(f_p)^p$. 
Therefore, $p^f\geq  (f_p)^p$ holds. 
\end{proof}

\subsection{Arc-transitive digraphs and group factorisations}\label{subsec:2.1}

Let $H$ be a finite group. If $H=AB$ for two subgroups $A$ and $B$ of $H$, then $H=AB$ is called a \emph{factorisation} of $H$. Such a factorisation is said to be \emph{proper} if both \(A\) and \(B\) are proper subgroups of \(H\), and \emph{homogeneous} if $A\cong B$.
The following result concerning group factorisations is elementary.

\begin{proposition}\label{pro:factorisation}
Let $H=AB$ be a factorisation of $H$.
Then  $\vert H\vert \vert A \cap B\vert= \vert A\vert \vert B\vert $, and  $H=A^xB^y$ for each $x,y\in H$. 
\end{proposition}

To investigate group factorisations \(H=AB\) satisfying certain conditions, we will make extensive use of the {\sc Magma}~\cite{Magma} function {\sf facsm} given below. 
This is a function that takes a group \(H\) and a positive integer \(m\) as input, and returns a complete set of conjugacy class representatives \((A, B)\) of subgroup pairs of \(H\) such that \(H = AB\) is a factorisation and  \(m\) divides \(|A|\) and \(|B|\). 

\begin{verbatim} 
facsm:=function(H,m) 
Re:=[];Ks:=Subgroups(H:OrderMultipleOf:=m); print "#Ks:",#Ks; 
while #Ks ge 2 do
 A:=Ks[1]`subgroup; Remove(~Ks,1); 
  for i in [1..#Ks] do
   B:=Ks[i]`subgroup;
   if #(A meet B)*#H eq #A*#B  then  
    Append(~Re,<A,B>); 
end if;	end for; end while; return Re; end function;
\end{verbatim}


Let \(H\) be a transitive permutation group on a set \(V\), and let \(v \in V\). An \(H_v\)-orbit in \(V\) is called an \emph{\(H\)-suborbit} relative to \(v\).  
For \(v_1 \in V \setminus \{v\}\), if there exists \(x \in H\) such that \((v, v_1)^x = (v_1, v)\), then the \(H\)-suborbit \(v_1^{H_v}\) is said to be \emph{self-paired}; otherwise, it is \emph{non-self-paired}. 
For a non-self-paired \(H\)-suborbit \(v_1^{H_v}\), the digraph with vertex set \(V\) and arc set \((v, v_1)^H\) is \(H\)-arc-transitive. Conversely, every \(H\)-arc-transitive digraph arises in this way. 
Note that the action of \(H\) on \(V\) is equivalent to its action on the right coset space \([H{:}H_v]\) by right multiplication. Under this correspondence, each \(H\)-suborbit corresponds to an \((H_v, H_v)\)-double coset \(H_v h H_v\), and it is non-self-paired if and only if \(h^{-1} \notin H_v h H_v\). 
These observations lead to the following useful criterion.

\begin{lemma}\label{lm:suborbit-doublecoset}
Let \(\Gamma\) be an \(H\)-arc-transitive digraph with an arc \((v, v_1)\), and let \(h \in H\) satisfy \(v^h = v_1\). Then the suborbit containing \(v_1\) is non-self-paired if and only if \(h^{-1} \notin H_v h H_v\).
\end{lemma}




The connection between $s$-arc-transitive digraphs and homogeneous factorisations of groups can be traced to a result of Giudici and Xia~\cite{GX2017}. For our purposes, the following special case of their result will be essential.

\begin{proposition}[{\cite[Lemma 2.2]{GX2017}}] \label{pro:homofac}
Let $\mathit{\Gamma}$ be an $H$-arc-transitive digraph, and let $ u\to  v\to v_1 $ be a $2$-arc of $\mathit{\Gamma}$.
Then  $\mathit{\Gamma}$ is $(H,2)$-arc-transitive if and only if  $H_v=H_{uv}H_{vv_1}$. 
\end{proposition} 

For an $(H,2)$-arc-transitive digraph $\Gamma$ with a $2$-arc $u \to v \to v_1$ and any normal subgroup $N$ of $H_v$, we obtain a factorisation $H_v/N=(H_{uv}N/N) (H_{vv_1}N/N)$ (this factorisation of $H_v/N$ may not be homogeneous). 
Studying such quotient factorisations is a basic tool for investigating homogeneous factorisations of $H_v$ itself. 
The  lemma below is based on  recent results on factorisations of almost simple groups~\cite{LWX2021+,LX2019,LX2021+}.
Recall that a group $X$ is said to be \emph{quasisimple} if $X=X'$ and $X/\mathrm{Z}(X)$ is nonabelian simple.
Note that, for a quasisimple group $X$, if a subgroup $Y$ contains the composition factor $X/\mathrm{Z}(X)$, then $Y$ must equal to $X$.

\begin{lemma}[{\cite[Lemma~3.3]{YFX2023}}]\label{lm:qsimple}
Let  $\mathit{\Gamma}$ be a connected $(H, s)$-arc-transitive digraph with $s\geq 2$, and let $u\to v \to v_1$ be a $2$-arc of $\mathit{\Gamma}$.  
Suppose that  $H_v^{(\infty)}$ is quasisimple.
Let $\mathrm{Rad}(H_v)$ be the largest solvable normal subgroup of $H_v$ and let $S$ be the socle  of $H_v/\mathrm{Rad}(H_v)$. 
\begin{enumerate}[\rm (a)]   
\item  The simple group  $S $ is  isomorphic to one of the following:
\begin{equation}\label{eq:groupsqs}
\mathrm{A}_6,\mathrm{M}_{12},\mathrm{Sp}_4(2^t),\mathrm{P\Omega}^{+}_8(q), \mathrm{PSL}_2(q),\mathrm{PSL}_3(3),\mathrm{PSL}_3(4),\mathrm{PSL}_3(8),\mathrm{PSU}_3(8),\mathrm{PSU}_4(2). 
\end{equation}  
Moreover, if $H_v$ is almost simple, then $S$ is isomorphic to $\mathrm{A}_6$, $\mathrm{M}_{12}$, $\mathrm{Sp}_4(2^t)$ or $\mathrm{P\Omega}^{+}_8(q)$ and $s\leq 2$.

\item Suppose that $S $ is one in~\eqref{eq:groupsqs}. 
Let $\overline{H_v}=H_v/\mathrm{Rad}(H_v)$, $\overline{H_{uv}}=H_{uv} \mathrm{Rad}(H_{ v})/\mathrm{Rad}(H_{ v})$, $\overline{H_{vv_{1}}}=H_{vv_{1}}\mathrm{Rad}(H_{ v})/\mathrm{Rad}(H_{ v})$.
Then both $\overline{H_{uv}}$ and $\overline{H_{vv_1}}$ are core-free in $\overline{H_{v}}$, and the factorisation $\overline{H_{v}}=\overline{H_{uv}}\,\overline{H_{vv_1}}$ satisfies
one of the following (interchanging $\overline{H_{uv}}$ and $\overline{H_{vv_1}}$ if necessary):
\begin{enumerate}[\rm (b.1)] 
\item $S=\mathrm{Sp}_4(2^f)$ with $f\geq 2$, $\overline{G_{v}}\leq\mathrm{\Gamma Sp}_4(q)$, and $(\overline{G_{uv}}\cap S,\overline{G_{vv_1}}\cap S)\cong(\mathrm{Sp}_2(q^2).2,\mathrm{Sp}_2(q^2))$ or $(\mathrm{Sp}_2(q^2).2,\mathrm{Sp}_2(q^2).2)$; 
\item $S=\mathrm{P\Omega}_8^{+}(q)$, $\overline{H_{v}}\leq  \mathrm{P\Gamma O}_8^{+}(q) $, and $(\overline{H_{uv}}\cap S,\overline{H_{vv_1}}\cap S)\cong (\Omega_7(q),\Omega_7(q))$.
 \item $S=\mathrm{PSL}_2(q)$, $\overline{H_{uv}}\cap S\leq \mathrm{D}_{2(q+1)/d}$ and $ \overline{H_{vv_1}}\cap S \leq q{:}((q-1)/d)$, where $d=(2,q-1)$.
\item $(\overline{H_{v}},\overline{H_{uv}},\overline{H_{vv_1}})$ satisfies Table \ref{tb:exceptionalfacs}, where $\mathcal{O}\leq \mathrm{C}_2$, and $\mathcal{O}_1$ and $\mathcal{O}_2$ are subgroups of $\mathcal{O}$ such that $\mathcal{O}=\mathcal{O}_1\mathcal{O}_2$.
\end{enumerate} 
\end{enumerate} 
\end{lemma}

\begin{table}[ht]
\centering
\caption{Some factorisations for $\overline{H}_v = \overline{H}_{uv} \overline{H}_{uv_1}$}
\label{tb:exceptionalfacs}
\begin{tabular}{lll}
\hline 
\noalign{\vspace{0.5ex}}
$\overline{H}_v$ & $\overline{H}_{uv}$ & $\overline{H}_{uv_1}$ \\
\hline
$\mathrm{A}_6.\mathcal{O} (\leq \mathrm{S}_6)$ & $\mathrm{A}_5.\mathcal{O}$ & $\mathrm{A}_5.\mathcal{O}$ \\
$\mathrm{M}_{12}$ & $\mathrm{M}_{11}$ & $\mathrm{M}_{11}$ \\
$\mathrm{PSL}_2(7).\mathcal{O}$ & $7, 7:(3 \times \mathcal{O})$ & $\mathrm{S}_4$ \\
$\mathrm{PSL}_2(11).\mathcal{O}$ & $11, 11:(5 \times \mathcal{O})$ & $\mathrm{A}_4$ \\
$\mathrm{PSL}_2(23).\mathcal{O}$ & $23:(3 \times \mathcal{O})$ & $\mathrm{S}_4$ \\
$\mathrm{PSL}_3(3).\mathcal{O}$ & $13, 13:(3 \times \mathcal{O})$ & $3^2.2.\mathrm{S}_4$ \\
$\mathrm{PSL}_3(3).\mathcal{O}$ & $13:(3 \times \mathcal{O})$ & $\mathrm{A\Gamma L}_1(9)$ \\
$\mathrm{PSL}_3(4).(\mathrm{S}_3 \times \mathcal{O})$ & $7:(3 \times \mathcal{O}).\mathrm{S}_3$ & $2^4:(3 \times \mathrm{D}_{10}).2$ \\
$\mathrm{PSL}_3(8).(3 \times \mathcal{O})$ & $57:(9 \times \mathcal{O}_1)$ & $2^{3+6}:7^2.(\mathrm{D}_{10} \times \mathcal{O}_2)$ \\
$\mathrm{PSU}_3(3).\mathcal{O}$ & $57:(9 \times \mathcal{O}_1)$ & $2^{3+6}:(63:3).\mathcal{O}_2$ \\
$\mathrm{PSU}_4(2).\mathcal{O}$ & $2^4.5$ & $3^{1+2}.2:(\mathrm{A}_4.\mathcal{O})$ \\
$\mathrm{PSU}_4(2).\mathcal{O}$ & $2^4.5$ & $3^{1+2}.2:(\mathrm{A}_4.\mathcal{O})$ \\
$\mathrm{PSU}_4(2).\mathcal{O}$ & $2^4:\mathrm{D}_{10}.\mathcal{O}_1$ & $3^{1+2}.2:(\mathrm{A}_4.\mathcal{O}_2)$ \\
$\mathrm{PSU}_4(2).2$ & $2^4.5:4$ & $3^{1+2}:\mathrm{S}_3, 3^3:(\mathrm{S}_3 \times \mathcal{O}),$ \\
& & $3^3:(\mathrm{A}_4 \times 2), 3^3:(\mathrm{S}_4 \times \mathcal{O})$ \\
$\mathrm{M}_{11}$ & $11:5$ & $\mathrm{M}_9.2$ \\
\hline
\end{tabular}
\end{table}

The proof of Lemma~\ref{lm:qsimple} relies on connectivity of the digraph, specifically, the following result. 

\begin{lemma}[{\cite[Lemma~2.14]{GLX2019}}]\label{lm:Hvnormal}
Let $\mathit{\Gamma}$ be a connected $H$-arc-transitive digraph with an arc $ v\to v_1 $, and let $h \in H$ such that $v^h = v_1$.
Then no nontrivial normal subgroup of $H_v$ is normalised by $h$.
\end{lemma}

Next, we collect or prove some results on arc-transitive digraphs.

\begin{proposition}[{\cite[Lemma~2.13]{GLX2019}}]\label{pro:valency}
For each vertex-primitive arc-transitive digraph $\mathit{\Gamma}$, either $\mathit{\Gamma}$ is a directed cycle of prime length
or $\mathit{\Gamma}$ has valency at least $3$.
\end{proposition}

\begin{proposition}[{\cite[Corollary 2.11]{GX2017}}]\label{pro:HsMs-1}
 Let $\mathit{\Gamma}$ be an $(H,s)$-arc-transitive digraph with $s\geq 2$, and let $M$ be a vertex-transitive normal subgroup of $H$. 
Then  $\mathit{\Gamma}$ is $(M,s-1)$-arc-transitive.
\end{proposition}

\begin{lemma}[{\cite[Lemma~3.2]{YFX2023}}]\label{lm:T2at?}
Let $H$ be an almost simple group with socle $L$, and let  $\mathit{\Gamma}$ be a connected $(H, s)$-arc-transitive digraph with $s\geq 2$, and let $u\to v \to v_1$ be a $2$-arc of $\mathit{\Gamma}$.
Let $t$ be the number of orbits of $L_{uv}$ on the out-neighbours of $v$. Then $|H|/|L|=|H_{v}|/|L_{v}|=|H_{uv}|/|L_{uv}|=t|H_{uvv_1}|/|L_{uvv_1}|$, and $t=|L_v|/|L_{uv}L_{vv_1}|$, and if $L_{uv}$ is conjugate to $L_{vv_1}$ in $L_v$, then $|L_v|$ divides $|H|/|L|$.

\end{lemma}

\begin{lemma}\label{lm:HvHuv}
 Let  $\mathit{\Gamma}$ be an $(H, s)$-arc-transitive digraph with $s\geq 2$, and let $u\to v \to v_1$ be a $2$-arc of $\mathit{\Gamma}$.
Then $\vert H_u\vert^{s-1}$ divides $\vert H_{uv}\vert^{s}$. 
\end{lemma}

\begin{proof}
Let $x$ be the valency of $\mathit{\Gamma}$ and let $\mathit{\Omega}$ be the set of $s$-arcs  starting at $u$.
Then $\vert \mathit{\Omega}\vert=x^s $. 
Since $\mathit{\Gamma}$ is $(H, s)$-arc-transitive, $H_u$ acts transitively on $\mathit{\Omega}$.
It follows that $x^s$ divides $\vert H_u\vert$.
Note that $x=\vert H_u\vert/\vert H_{uv}\vert$ as $\mathit{\Gamma}$ is $H$-arc-transitive (see Proposition~\ref{pro:HsMs-1}).
Thus, $\vert H_u\vert^{s-1}$ divides $\vert H_{uv}\vert^{s}$.  
\end{proof}

In the following, $M_1 \circ M_2$ denotes a central product of $M_1$ and $M_2$ (so $[M_1, M_2]=1$ and $M_1 \cap M_2 \subseteq Z(M_1 M_2)$). 

\begin{lemma}\label{lm:N1N2corefree}
 Let  $\mathit{\Gamma}$ be a connected $(H, s)$-arc-transitive digraph with $s\geq 2$, and let $u\to v \to v_1$ be a $2$-arc of $\mathit{\Gamma}$.
Suppose that
 \(H_{v}^{(\infty)}=M_{1}\times M_{2}\) or \(M_{1}\circ M_{2}\), where $M_i$ is a quasisimple group such that $M_i/\mathrm{Z}(M_i)\cong N_i$ for each $i\in \{1,2\}$, and that $N_1$ has no section isomorphic to $N_2$. 
Let $N$ be a normal subgroup of $H_v$ such that $H_v/N$ is an almost simple with socle $N_2$.
Then in the factorisation $H_v/N=(H_{uv}N/N)(H_{vv_1}N/N)$,  both two factors $H_{uv}N/N$ and $H_{vv_1}N/N$ are core-free in $H_v/N$. 
\end{lemma}

\begin{proof}
Suppose  without loss of generality that $H_{uv}N/N$  is not core-free in $H_v/N$.
Then $H_{uv}$ contains the composition factor $N_2$.
Since $\mathit{\Gamma}$ is $H$-arc-transitive, there is an element   $h\in H$ such that  $u^h=v$ and $v^h=v_1$.
Then $H_{vv_1}=H_{uv}^h$ also contains the composition factor $N_2$.
Now $H_v^{(\infty)} = M_1 M_2$, where $M_1$ and $M_2$ are quasisimple and $M_2/\mathrm{Z}(M_2) \cong N_2$. Because $N_1$ has no section isomorphic to $N_2$, the only subgroup of $H_v^{(\infty)}$ that can cover $N_2$ is $M_2$ itself.  
Therefore, $M_2 $ is contained in both $ H_{uv}$ and $H_{vv_1}$. 
Then $M_2 $ is normalised by $h$, contradicting Lemma~\ref{lm:Hvnormal}. 
Therefore, $H_{uv}N/N$ must be core-free; by symmetry, $H_{vv_1}N/N$ is also core‑free. 
\end{proof}

The next result was proved in Chen's thesis \cite[Lemma~2.31]{Chen2023}. For the convenience of readers, we provide a proof here as well.

\begin{lemma}\label{lm:m2.o}
Let  $\mathit{\Gamma}$ be a connected $(H, s)$-arc-transitive digraph with $s\geq 2$, and let $u\to v \to v_1$ be a $2$-arc of $\mathit{\Gamma}$.
Suppose that $L$ is a vertex-transitive normal subgroup of $H$ and  $L_{v}\cong \mathrm{C}_{m}^2.\mathcal{O}$.
If there exists a prime divisor $r$ of $m$ with $\vert \mathcal{O}\vert_{r}=1$ and $\vert H/L\vert_{r}<m_{r}$, then $s\leq 2$. 
\end{lemma}
\begin{proof}
Suppose for a contradiction that $s\geq 3$. 
By Proposition~\ref{pro:HsMs-1}, $\mathit{\Gamma}$ is $(L,2)$-arc-transitive and $L_{v}$ admits a homogeneous factorisation $L_{v} \cong L_{uv}L_{vv_{1}}$.
Since $L_{v}\cong \mathrm{C}_{m}^2.\mathcal{O}$ and $\vert \mathcal{O}\vert_r=1$, we see that $L_{v}$ has a unique subgroup isomorphic to $ \mathrm{C}_{r}^2$; denote this subgroup by $R$. 
Write $m_r=r^a$. Then $\vert L_v\vert_r=r^{2a}$.
Since $L$ is a vertex-transitive, we have $H/L=LH_v/L\cong H_v/(H_v\cap L)=H_v/L_v$.
Then $\vert H_v\vert_r =\vert L_v\vert_r\cdot \vert H/L\vert_r <r^{3a}$.

We claim that $\vert L_{uv}\vert_r>r^a$. 
If $\vert L_{uv}\vert_r\leq r^a$, then the valency $x $ of $\mathit{\Gamma}$ is divisible by $r^a$ (note that $x=\vert L_v\vert/\vert L_{uv}\vert$).
Since  $\mathit{\Gamma}$ is  $(H,3)$-arc-transitive, $\vert H_v\vert$ is divisible by $x^3$ (see the proof of Lemma~\ref{lm:HvHuv}), hence by $r^{3a}$.  This  contradicts $\vert H_v\vert_r<r^{3a}$.

From the claim $\vert L_{uv}\vert_r>r^a$ and the structure $L_{v} \cong  \mathrm{C}_m^2.\mathcal{O}$ with $\vert \mathcal{O}\vert_r=1$, it follows that $L_{uv}$ must contain  the subgroup $R$. 
Let $h\in L$ such that $u^h=v$ and $v^h=v_1$. Then $L_{vv_{1}}=L_{uv}^{h}$ also contains $R$.
Thus, $R^h=R$, contradicting Lemma \ref{lm:Hvnormal}. 
\end{proof}

The technique used in the proof of the following lemma originates in~\cite{CGP2023}.
\begin{lemma}
\label{lm:m.o}
Let  $\mathit{\Gamma}$ be a connected $(H, s)$-arc-transitive digraph with $s\geq 1$, and let $u\to v \to v_1$ be a $2$-arc of $\mathit{\Gamma}$.
Suppose that $L$ is a vertex-transitive normal subgroup of $H$ and  $L_{v}\cong \mathrm{C}_{m}.\mathcal{O}$.
If there exists a prime divisor $r$ of $m$ such that  $\vert \mathcal{O}\vert_{r}=1$ and $\vert H/L\vert_{r}<m_{r}$, then $s\leq 1$. 
\end{lemma}

\begin{proof} 
Suppose for a contradiction that $s\geq 2$. Let $h\in H$ such that $u^h=v$ and $v^h=v_1$. By Proposition~\ref{pro:homofac}, $H_{v}$ admits a homogeneous factorisation $H_{v}=H_{uv}H_{vv_{1}}$ with  $H_{uv}^{h}=H_{vv_{1}}$. 
Write $m_r=r^{a}$ and $\vert H/L\vert_{r}=r^b$. 
Note that $\vert H_v/L_v\vert=\vert H/L\vert$ as $L$ is  transitive on the vertex set.
Then $a>b$ and $\vert H_v\vert_r\leq r^{a+b}$.
It follows from $H_{v}=H_{uv}H_{vv_{1}}$ that $\vert H_{uv}\vert_r\geq r^{(a+b)/2}>r^b$. 
This implies that $\vert H_{uv} \cap \mathrm{C}_m\vert_r>1$, and hence $H_{uv}$ contains the unique subgroup $R\cong \mathrm{C}_r$ of $\mathrm{C}_m$. 
Since $\vert H_{vv_1}\vert_r= \vert H_{uv}\vert_r >r^b$,  the group  $H_{vv_{1}}$ also contains $R$.
Then $R^h=R$, contradicting Lemma \ref{lm:Hvnormal}.
\end{proof}

\subsection{Subgroups of maximal rank}\label{subsec:torus}
Among the maximal subgroups of almost simple exceptional groups of Lie type, an important family is formed by subgroups of maximal rank. Typical examples include $(\mathrm{Sp}_4(q)\times \mathrm{Sp}_4(q)).2$ and $\mathrm{C}_{q+1}^4.W(F_4)$ in $F_4(q)$.
When  studying homogeneous factorisations of such subgroups, an explicit description of their structure is essential. In particular, we need to understand how the group $2$ acts on $\mathrm{Sp}_4(q)\times \mathrm{Sp}_4(q)$ and how $W(F_4)$ (the Weyl group of $F_4(q)$) acts on $\mathrm{C}_{q+1}^4$. 
We will adopt a computational approach to  analyse these group actions.
In this subsection, we fix notation (following Liebeck, Saxl, and Seitz~\cite{LSS1992}) and collect some background results about subgroups of maximal rank.

Let $G$  be a simple adjoint algebraic group over an algebraically closed field  of characteristic $p>0$ and let $\sigma : G \to G$ be a Steinberg endomorphism  
such that $L=(G_{\sigma})'$ is a finite simple group of Lie type over $\mathbb{F}_q$, where $q=p^a$ and $G_{\sigma}=\{g:g\in G\mid g^{\sigma}=g\}$.
Let $H$ be an almost simple group with socle $L$, and let $D$ be a $\sigma$-stable closed connected reductive subgroup of $G$ that contains a $\sigma$-stable maximal torus $T$ of $G$.
Then $N_H(D)$ is called a \emph{subgroup of maximal rank} in $H$. 

Now suppose that $L$ is one of $F_4(q)$, $E_6(q)$ or ${}^2E_6(q)$;  then $G_{\sigma}=\mathrm{Inndiag}(L)=F_4(q)$, $E_6(q).(3,q-1)$ or ${}^2E_6(q).(3,q+1)$, respectively. 
Then $G$ contains a  $\sigma$-stable maximal torus $S$ on which $\sigma$ acts as $s\mapsto s^q$ if $L\neq {}^2E_6(q)$, and as $s\mapsto s^{-q}$  if $L={}^2E_6(q)$.
Let $W=N_G(S)/S$ be the  Weyl group of $G$.

Because $T$ is also a $\sigma$-stable maximal torus  of $G$, there exists  $g\in G$ such that $T=S^g$. 
Set $E=D^{g^{-1}}$, and let $\Phi$ and $\Delta$ be the root systems of $G$ and $E$, (respectively) with respect to $S$. 
Denote by $W(\Delta)$ the Weyl group of $E$, and put $W_{\Delta}=N_W( \Delta )/W(\Delta)$, where $N_W( \Delta )$ is the normaliser of $W(\Delta)$ in $W$.
Since both $D=E^g$ and $T=S^g$ are $\sigma$-stable, $g^{\sigma}g^{-1} \in  N_G(S)\cap N_G(E)$, and it maps to an element of $W_{\Delta}$. 
Moreover, by the choice of the action of $\sigma$ on $S$, the $G_{\sigma }$-classes of $\sigma$-stable $G$-conjugates of $E$ are in one-to-one correspondence with the conjugacy classes of $W_{\Delta}$(see~\cite[Lemma~1.1]{LSS1992}; for a general statement see Carter~\cite[Corollary~3]{Carter1978}).

Regard  $w:=g^{\sigma}g^{-1}$ as an element of $W$. 
Then 
\[
(\sigma w)^g=\sigma,\ (G_{\sigma w})^g=G_{\sigma},\ (S_{\sigma w})^g=T_{\sigma},\  (E_{\sigma w})^g=D_{\sigma}.
\]
Consequently, $N_{G_{\sigma}}(D)/D_{\sigma}=N_{G_{\sigma w}^g}(E^g)/E_{\sigma w}^g \cong N_{G_{\sigma w}}(E)/E_{\sigma w}$. 
By~\cite[Lemma~1.2]{LSS1992}, 
\begin{equation}\label{eq:CWdelta}
N_{G_{\sigma w}}(E)/E_{\sigma w}\cong C_{W_{\Delta}}(wW({\Delta})),
\end{equation} 
where $C_{W_{\Delta}}(wW({\Delta}))$ denotes the centraliser of the coset $wW({\Delta})$ in $W_{\Delta}$.

\subsubsection{The case $E>S$}
Suppose that $E>S$. 
To determine the Lie type of the finite group $E_{\sigma w}$, we analyse the action of $\sigma w$ on $\Delta$, which follows the standard procedure of identifying fixed-point subgroups under Steinberg endomorphisms; see, for example, \cite[Sections 22.1 and 22.2]{GD2011}.
For more information of the group structure of \(E_{\sigma w}\) we refer to~\cite[Proposition~2.6.2]{CFSG}.
Moreover,  Carter~\cite[Corollary~5]{Carter1978} tells us that $W_{\Delta}$ is isomorphic to the group of symmetries induced by $W$ on the Dynkin diagram of $\Delta$.
  
Here we give an illustrative example. Let $G=F_4$ and let Dynkin diagram of $L=F_4(q)$ be as follows, where $\{ \alpha_1,\alpha_2,\alpha_3,\alpha_4\}$ are simple roots:
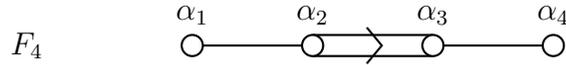
\begin{figure}[h]
\begin{center}
\begin{tikzpicture}[scale=0.4]  
\draw [thick] (2,2.85) -- (-2,2.85) ;
\draw [thick] (2,2.15) -- (-2,2.15) ;
\draw [thick] (-0.2,3.1) -- (0.3,2.5) -- (-0.2,1.9) ;
\draw [thick] (2,2.5) -- (6,2.5) ;
\draw [thick] (-2,2.5) -- (-6,2.5) ;  
\filldraw[thick,fill=white] (2,2.5) circle (10pt);
\node at (2,3.5) {\normalsize $\alpha_3$}; 
\filldraw[thick,fill=white] (-2,2.5) circle (10pt);
\node at (-2,3.5) {\normalsize $\alpha_2$};
\filldraw[thick,fill=white] (6,2.5) circle (10pt);
\node at (6,3.5) {\normalsize $\alpha_4$};
\filldraw[thick,fill=white] (-6,2.5) circle (10pt);
\node at (-6,3.5) {\normalsize $\alpha_1$}; 
\node at (-11.5,2.5) { \normalsize $F_4$}; 
\end{tikzpicture}
\end{center}
\end{figure}

 When $q$ is even, $(\mathrm{Sp}_4(q)\times \mathrm{Sp}_4(q)).2$ and $\mathrm{Sp}_4(q^2).2$ are two subgroups of maximal rank of $L$; they arise from a subsystem $\Delta$ of type $2C_2$. Such a   subsystem $\Delta$ can be chosen with simple roots $\{\alpha_0,-\beta_0,\alpha_3,\alpha_2  \}$, where $\alpha_0=2\alpha_1+3\alpha_2+4\alpha_3+2\alpha_4$ is the  highest root of $\Phi$, and  $\beta_0=\alpha_1+2\alpha_2+3\alpha_3+2\alpha_4$ is the  highest root of  a subsystem of typle $C_4$. 
We  locate $\Delta$ through the following steps: 
\begin{enumerate}[(1)]
  \item By~\cite[Theorem B.18]{GD2011}, \(F_4\) contains a maximal closed subsystem of type \(B_4\) with simple roots \(\{-\alpha_0,\alpha_1,\alpha_2,\alpha_3\}\).
  \item Applying the graph automorphism \(\tau\) of \(F_4\) that interchanges \(\alpha_1 \leftrightarrow \alpha_4\) and \(\alpha_2 \leftrightarrow \alpha_3\), the image of the above \(B_4\) is a subsystem of type \(C_4\)  (see~\cite[Example~B.22(4)]{GD2011}).  A straightforward computation yields \((-\alpha_0)^\tau = -\beta_0\); hence this \(C_4\) subsystem has simple roots \(\{-\beta_0,\alpha_4,\alpha_3,\alpha_2\}\).
  \item Using~\cite[Theorem B.18]{GD2011} again, this \(C_4\) subsystem possesses a maximal closed subsystem of type \(2C_2\).  A convenient choice is the one with simple roots \(\{\alpha_0,-\beta_0,\alpha_3,\alpha_2\}\).  In particular, \(\{\alpha_2,\alpha_3\}\) and \(\{\alpha_0,-\beta_0\}\) are the simple roots of the two \(C_2\) factors.
\end{enumerate}

We now use the following {\sc Magma}~\cite{Magma} code to determine $W_{\Delta}$ and identify which Weyl element $w$ gives rise to $(\mathrm{Sp}_4(q)\times \mathrm{Sp}_4(q)).2$ or $\mathrm{Sp}_4(q^2).2$:
\begin{verbatim}
R:=RootSystem("F4":BaseField:=GF(2)); W:=CoxeterGroup(R);Rs:=Roots(R);  
a0:=(2*Rs[1]+3*Rs[2]+4*Rs[3]+2*Rs[4]); Position(Rs, a0); //24  
b0:=(1*Rs[1]+2*Rs[2]+3*Rs[3]+2*Rs[4]);Position(Rs, -1*b0);//45
SD:={2,3,24,45}; RD:=sub<R|SD>;RD; WD:=sub<W|[ReflectionPermutation(W,i):i in SD]>;
NWD:=Normaliser(W,WD); #NWD div #WD;qNWD:=Stabiliser(NWD,SD);qNWD;
\end{verbatim}
In the code, {\rm\texttt{W}} and  {\rm\texttt{WD}} correspond to $W$ and $W(\Delta)$, respectively, and  {\rm\texttt{NWD}} represents $N_{W}(\Delta)$.
The roots $\alpha_2,\alpha_3,\alpha_0,-\beta_0$ are indexed as $2,3,24,45$, respectively.
Since $W_{\Delta}$ is isomorphic to the group of symmetries induced by $W$ on the Dynkin diagram $\Delta$, it must stabilise (setwise)  the set {\rm\texttt{SD}} of simple roots of $\Delta$, and hence we can obtain $W_{\Delta}$ by computing the setwise stabiliser of {\rm\texttt{SD}} in {\rm\texttt{NWD}}; this stabiliser is  {\rm\texttt{qNW}}.
The computation shows that $W_{\Delta}$ is generated by a single involution $\pi$, whose permutation representation is
\begin{verbatim}
(1,25)(2,24)(3,45)(5,23)(6,8)(7,43)(9,40)(10,12)(11,42)(13,37)(15,39)(16,33)(18,35)
(19,31)(21,27)(22,46)(26,48)(29,47)(30,32)(34,36). 
\end{verbatim}
In terms of the simple roots, $\pi$ acts as $(\alpha_2,\alpha_0)(\alpha_3,-\beta_0)$; therefore it swaps the two $C_2$ subsystems.
By analysing the root system of $E_{\sigma w}$ using~\cite[Section 23.1]{GD2011}, we obitain, if $w=1$ then $E_{\sigma w}= \mathrm{Sp}_4(q)\times \mathrm{Sp}_4(q) $, and if $w=\pi$ the $E_{\sigma w}= \mathrm{Sp}_4(q^2)$ (a set of simple roots can be choosen to be $\{(\alpha_2+\alpha_0)/2, (\alpha_3-\beta_0)/2\}$).
Clearly, $C_{W_{\Delta}}(wW({\Delta}))=W_{\Delta}\cong \mathrm{C}_2$ swaps two $\mathrm{Sp}_4(q)$ factors in the group $(\mathrm{Sp}_4(q)\times \mathrm{Sp}_4(q)).2$.

\subsubsection{The case $E=S$}
Suppose that $E=S$. 
Then  $W_{\Delta}=W$ and so~\eqref{eq:CWdelta} is reduced to
\begin{equation}\label{eq:CWw}
N_{G_{\sigma w}}(S)/S_{\sigma w}\cong C_{W}(w). 
\end{equation}

Given a subgroup of maximal rank arising from this case, such as $(q+1)^4.W(F_4)$, we aim to determine the corresponding Weyl element $w$, and then compute the group structure of $C_W(w)$, and analyse the action of $C_W(w)$ on $S_{\sigma w}$. 

Our first task is to determine the Weyl element \(w\) corresponding to a given order polynomial \(|S_{\sigma w}|\) (such as \((q+1)^4\)). Although these  pairs $(|S_{\sigma w}|, w)$ are tabulated in the literature for the groups under consideration (e.g., \cite{Haller2005}), we determine \(w\) by direct computation in {\sc Magma} for two key reasons. First, this is necessary for \({}^2E_6(q)\), because the endomorphism assumed in the relevant table \cite[Table A.8]{Haller2005} differs from our endomorphism \(\sigma\). Second, computing directly provides a  uniform  method: it works consistently for all types with our specific \(\sigma\), and it explicitly yields the matrix representing \(w\) and its centraliser \(C_W(w)\)—precisely the data needed for the subsequent analysis of the action of \(C_W(w)\) on \(S_{\sigma w}\). The computation itself is efficient and integral to our computational pipeline. 

To  determine $w$  from a given polynomial $|S_{\sigma w}|$, we use the following characterisation. 
Let $X$ and $Y$ be the root space and the coroot space with respect to $S$, respectively. 
By~\cite[Proposition 25.3]{GD2011}, $|(S^g)_\sigma|=|\mathrm{det}_{X\otimes \mathbb{R}}(\sigma w^{-1}-1)|$.
Notice that $|S_{\sigma w}|=|T_\sigma|=|(S^g)_\sigma|$.
Let $N$ be the matrix of $w^{-1}$ acting on the root space $X$. Then the choice of   $\sigma$ leads to the formula: 
\begin{equation}\label{eq:Ssigmaw}
|S_{\sigma w}|=\begin{cases}
|\det(qN -1)| & \text{ if $L=F_4(q)$ or $E_6(q)$},\\
|\det(-qN -1)| & \text{ if $L={}^2E_6(q)$}.\\
\end{cases}
\end{equation} 

We implement this search in {\sc Magma}. 
For example, for $G=F_4$ we compute as follows:
\begin{verbatim}
RD:=RootDatum("F4":Isogeny := "Ad"); W:=CoxeterGroup(GrpPermCox,RD);
 _,f:= ReflectionGroup(W); _,m:=CoxeterGroup(GrpFPCox, W);
P<q> := PolynomialRing(RationalField());
for c in Classes(W) do
w:=c[3]; N:= Matrix(P,f(w^-1)) ; <w,m(w),factorisation(Determinant(q*N-1))>;
end for;
\end{verbatim}
Here, \texttt{W} represents the Weyl group $W$ as a permutation group, while \texttt{f} and \texttt{m} are maps converting elements of \texttt{W} into matrix group and finitely presented group forms, respectively. 
Running this code shows  that the Weyl element $w$ yielding $|S_{\sigma w}|=(q+1)^4$ is the longest element of $W$. This element acts as $-1$ on the root system $\Phi$ and lies in the centre of $W$. Notably, for this particular $w$, the matrix $qN - 1$ is scalar with entry $-q - 1$, which implies that $|S_{\sigma w}| = (q+1)^4$.

To analyse the action of $C_W(w)$ on $S_{\sigma w}$ computationally, we employ an explicit developed by~Cohen–Murray–Taylor~\cite{CMT2004}, and Haller~\cite{Haller2005}. 
Note that $\sigma w$ induces a permutation $\rho$ on  $\Phi$ (see~\cite[Proposition 22.2]{GD2011}).
Let $r$ be the order of $\rho$,  $\ell$  the rank of $G$, and set  $K=\mathbb{F}_{q^{r}}$. 
By~\cite[Sections~5.3~and~5.4]{Haller2005},  $S_{\sigma w}$ can be identified with  a subgroup of the standard split torus of $G(K)$  (which has the same type as $G_{\sigma }$ but defined over  the larger field $K$).
By~\cite[Section 4.1]{CMT2004}, the standard split torus of $G(K)$ can be represented as $Y \otimes K^{\times}$.  
Fix a basis $f_1,f_2,\ldots,f_\ell$  of $Y$.
According to~\cite[Section 5.2]{CMT2004}, every element of $Y \otimes K^{\times}$ can be uniquely written as $\prod_{i=1}^{\ell}f_i  \otimes h_i$ where $h_i\in K^{\times}$, and is represented by the vector  $(h_1,h_2,\ldots,h_\ell)$.
Let $\xi$ be a primitive element of $K$ and write each $h_i$ as $\xi^a_i$ with $a_i\in \mathrm{C}_{q^{r}-1}$.
By~\cite[Section 5.2]{CMT2004}, an automorphism of  $Y$ with matrix $M$ induces an automorphism of $Y \otimes K^{\times}$ By
 \[\prod_{i=1}^{\ell} f_i\otimes \xi^{a_i} \mapsto  \prod_{j=1}^{\ell} f_j\otimes (\prod_{i=1}^{\ell}\xi^{a_iM_{ij}}) ,\] 
 which, in vector notation, becomes 
 \begin{equation}\label{eq:xJ}
 (\xi^{a_1},\ldots,\xi^{a_{\ell}}) \mapsto (\xi^{b_1},\ldots,\xi^{b_{\ell}}) \text{ with } (b_1,\ldots,b_\ell)=(a_1,\ldots,a_\ell)M.    
\end{equation}  
Now let $M_0$ the matrix of $w$ acting on  $Y$, and set $M_1=M_0$ if $L=F_4(q)$ or $E_6(q)$, or $-M_0$ if $L={}^2E_6(q)$.
By~\cite[Theorem 5.6]{Haller2005}, the torus $S_{\sigma w}$ consists precisely of those vectors whose exponent vectors satisfy a linear condition:
\begin{equation}\label{eq:TwF}
S_{\sigma w} =\{(\xi^{a_1},\ldots,\xi^{a_{\ell}}) \mid (a_1,\ldots,a_\ell) \in \mathrm{C}_{q^{r}-1}^{\ell},(a_1,\ldots,a_\ell)(qM_1 -1)=0\}.
\end{equation}

With equations Eq.~\eqref{eq:xJ} and Eq.~\eqref{eq:TwF} we can analyse the action of $C_W(w)$ on $S_{\sigma w}$ computationally in {\sc Magma}.
For example,  take $G=F_4$ and let $w$ be the longest element of $W$. Since $\sigma$ acts trivially on the roots, it follows that $\sigma w$ induces the same permutation  as $w$ on  $\Phi$; hence $r=2$ and $K=\mathbb{F}_{q^2}$.
 Therefore,  
\[
\begin{split}    
S_{\sigma w} &=\{(\xi^{a_1},\ldots,\xi^{a_{4}}) \mid (a_1,\ldots,a_4) \in \mathrm{C}_{q^2-1}^{4},(a_1,\ldots,a_4)(-q -1)=0\},\\
&=\{(\xi^{(q-1)c_1},\ldots,\xi^{(q-1)c_{4}}) \mid (c_1,\ldots,c_4) \in \mathrm{C}_{q+1}^{4}\}.
\end{split}
\]
Equation \eqref{eq:xJ} shows that if an element of $C_W(w)$ acts on $Y$  via a matrix $M$, then its action on $S_{\sigma w}$ corresponds to the natural action of  $M$ on $\mathrm{C}_{q+1}^{4}$. We do the following {\sc Magma} computation to get the natural module of $C_W(w)$:
\begin{verbatim}
RD:=RootDatum("F4":Isogeny := "Ad"); W:=CoxeterGroup(GrpPermCox,RD);
 _,h:= CoreflectionGroup(W); w:=LongestElement(W); 
Cw:=Centraliser(W,w); V:=GModule(h(Cw));
\end{verbatim}
Here, the {\sc Magma} command {\rm\texttt{CoreflectionGroup}} gives the action of $W $ on the coroot space $Y$, and {\rm\texttt{GModule}} constructs the natural module for this matrix group $C_W(w)$.

\section{Proof of Theorem~\ref{th:main}}

 Throughout this section, we work under the following hypothesis.

\begin{hypothesis}\label{hy:1}
Let $H$ be an almost simple group with socle  \[L\in \{ {}^3D_4(q), G_2(q), {}^2F_4(q),F_4(q),E_6(q), {}^2E_6(q),G_2(2)',{}^2F_4(2)'\} .\]  
Let  $\mathit{\Gamma}$ be a connected $H$-vertex-primitive $(H, s)$-arc-transitive digraph with $s\geq 2$. 
Let $  u \rightarrow v \rightarrow v_1 $ be a $2$-arc in $\mathit{\Gamma}$ and let $h\in L$ such that $ u^h=v $ and $ v^h=v_1 $. 
Write $q = p^f$, where $p$ is prime and $f$ a positive integer.
\end{hypothesis}

According to Proposition~\ref{pro:HsMs-1}, $L$ is transitive on the vertex set of  $\mathit{\Gamma}$.
Hence, $|L|\geq |V(\mathit{\Gamma})|$.
Note that $|G_2(2)'| = 6,048$ and $|{}^2F_4(2)'|= 17,971,200$, while the smallest vertex‑primitive $2$-arc‑transitive digraph  has order $30,758,154,560$ (see \cite{YFX2023}). Hence, $L$ cannot be $G_2(2)'$ or ${}^2F_4(2)'$.

Since $H$ is primitive on the vertex set of $\mathit{\Gamma} $, the vertex stabiliser  $H_v$ is a maximal subgroup of $H$; it is core-free because the action on $V(\mathit{\Gamma})$ is faithful.
Lists of core-free maximal subgroups of the almost simple group $H$ are available in the literature, see the introduction section.
 
By Lemma~\ref{lm:qsimple}(a), if $H_v$ is an almost simple group, then $s\leq 2$ and its socle must be isomorphic to $\mathrm{A}_6$, $\mathrm{M}_{12}$, $\mathrm{Sp}_4(2^f)$ or $\mathrm{P\Omega}_8^{+}(q)$.
Furthermore, if the socle of $H_v$ is $\mathrm{Sp}_4(2^f)$ or $\mathrm{P\Omega}_8^{+}(q)$, then by Lemma~\ref{lm:qsimple}(b.1) and (b.2), $H_v$ contains no graph automorphism of $\mathrm{Sp}_4(2^f)$, or a triality graph automorphism $\mathrm{P\Omega}_8^{+}(q)$, respectively.
 
The proof proceeds as follows. 
In Subsection~\ref{subsec:parabolic}, we prove that under Hypothesis~\ref{hy:1}, $H_v$ is not a parabolic subgroup. 
The remaining cases are treated in Subsections~\ref{sec:3D4q}--\ref{sec:E6q2E6q}, where we handle separately the groups with socle ${}^3D_4(q)$, $G_2(q)$, ${}^2F_4(q)$, $F_4(q)$, $E_6(q)$, or ${}^2E_6(q)$.

\subsection{Parabolic subgroups}\label{subsec:parabolic}
In this subsection, we prove that, under Hypothesis~\ref{hy:1}, $H_v$ can not be a parabolic subgroup.  

Suppose for a contradiction that $H_v$ isparabolic. 
Then $L_v$ is a parabolic subgroup of $L$. 
According to Proposition~\ref{pro:HsMs-1}, $\mathit{\Gamma}$ is $L$-arc-transitive.
By Lemma~\ref{lm:suborbit-doublecoset}, $v_1^{L_v}$ is  a non-self-paired $L$-suborbit, and the double coset $L_vhL_v$ satisfies $h^{-1}\notin L_vhL_v$.
The $L$-suborbits are well understood by the theory of $(B,N)$-pairs; for a comprehensive treatment we refer to~\cite[Chapter~10]{BCN-book} or~\cite[Chapter~2]{Carter1985}. 
Below we recall only the facts needed for our computation.

Since $L$ is a group of Lie type, it has a $(B,N)$-pair, where $B$ denotes a Borel subgroup. Let $W=N/(B\cap N)$ be the Weyl group of $L$, $\Phi$ the set of roots and $\Pi$ the set of simple roots.. 
Let $J \subseteq \Pi$ be such that  $L_v$ is the parabolic subgroup $P_J$ corresponding to $J$, and let $W_J$ be the  associated parabolic subgroup of $W$.
There is a bijection between $(W_J,W_J)$-double  cosets in $W$ and $(P_J,P_J)$-double  cosets in $L$:  $W_JwW_J \leftrightarrow P_J\dot{w}P_J $, where $ \dot{w} $ is a preimage of $w$ in $N$. 
The representation $w$ of the double coset $W_JwW_J$ can be uniquely chosen so that it is the unique element in $W_JwW_J$ with the shortest length (see~\cite[Proposition 2.7.5]{Carter1985}). Denote by  $D_{J,J}$ the set of all such minimal-length representatives. 
Due to this  uniqueness, for any $w\in D_{J,J}$, $w^{-1}\in W_JwW_J$ if and only if $w$ is the identity or an involution in $W$. 

Suppose that $\mathrm{\Gamma}$  arises from the double coset $P_J\dot{w} P_J$  in $L$, where $w \in D_{J,J}$.
Then $L_{v} \cong P_J$, $L_vhL_{v} \cong P_J\dot{w} P_J$, and  the condition $h^{-1} \notin P_J\dot{w} P_J$ forces  $w$ to have order different from $1$ and $2$.
From the Weyl group and root data one can compute $ \vert P_J \cap P_J^{\dot{w}}\vert $ and   $ \vert P_J \cap P_J^{\dot{w}} \cap P_J^{\dot{w}^{-1}} \vert_p $, which equal $\vert L_{vv_{1}}\vert $ and $\vert L_{uvv_{1}}\vert_p$, respectively. 
By Proposition~\ref{pro:homofac},  $H_v=H_{uv}H_{vv_{1}}$, and so $\vert H_v\vert \cdot \vert H_{uvv_{1}}\vert  =\vert H_{uv} \vert \cdot \vert H_{vv_{1}}\vert $. Taking $p$-parts we obtain 
\begin{equation}\label{eq:Gvp}
\vert H_v\vert_p \vert H_{uvv_{1}}\vert_p =\vert H_{uv} \vert_p \vert H_{vv_{1}}\vert_p.
\end{equation}
Thus, we proceed as follows: determine all elements $w \in D_{J,J}$ that are not involutions (hence correspond to non‑self‑paired $L$-suborbits), and then for each such $w$, verify whether condition~\eqref{eq:Gvp} can be satisfied.

For example, take $L = F_4(q)$ and assume that $L_v$ is a maximal parabolic subgroup of type $A_1(q)A_2(q)$.
Let $\Pi = {\alpha_1, \alpha_2, \alpha_3, \alpha_4}$, with Dynkin diagram as in Subsection~\ref{subsec:torus}.
 
Then $L_v = P_J$ where $J$ is either $\{\alpha_1, \alpha_2, \alpha_4\}$ or $\{\alpha_1, \alpha_3, \alpha_4\}$.
Consider the case $J = \{\alpha_1, \alpha_2, \alpha_4\}$.
The following {\sc Magma} code computes the lengths of the $L$-suborbits and determines whether each is self-paired:

\begin{verbatim}
W0:=CoxeterGroup(GrpFPCox, "F4"); W,phi := CoxeterGroup(GrpPermCox, W0); 
J:={1,2,4}; DJ,_:=Transversal(W0, J, J); phiDJ:=[phi(x): x in DJ];  
print "the rank is",#DJ; RJ:= Transversal(W0, J); 
WJ:= StandardParabolicSubgroup(W, J);  F<q> := PolynomialRing (RationalField()); 
for idw in [1..#phiDJ] do
w:=phiDJ[idw];  RJWw:=[x: x in RJ|w eq TransversalElt(W,WJ,phi(x),WJ)];  
fw:=&+[q^j: j in [Length(x): x in RJWw]]; paired:=Order(w) in {1,2}; 
print idw,"-th suborbit, is self-paired:", paired," length is:",fw;  
end for;  
\end{verbatim} 
The code uses \texttt{Transversal(W0, J, J)} to obtain $D_{J,J}$ (minimal-length representatives of $(W_J,W_J)$-double cosets) and \texttt{Transversal(W0, J)} for a transversal of $W_J$ in  $ W$. For each $w \in D_{J,J}$, the list \texttt{RJWw} records the right cosets contained in $W_J w W_J$. The polynomial $  \sum_{x \in \texttt{RJWw}} q^{\ell(x)}$ yields the length of the corresponding $L$-suborbit, and the suborbit is self-paired exactly when the order of $w$ is $1$ or $2$. Running the code reveals four non-self-paired suborbits: two of length $q^7(q^2+q+1)(q+1)$ and two of length $q^{10}(q^2+q+1)(q+1)$.

Suppose that $\mathit{\Gamma}$ arises from a  non-self-paired $T$-suborbit of length $q^7(q^2+q+1)(q+1)$. 
Note that $\vert L_v\vert_p=\vert L\vert_p=q^{24}$.
From the suborbit length we have $\vert  L_{uv}\vert_p=\vert L_v\vert_p/q^7=q^{17}$. 
According to~\cite[Proposition~2.5.9]{Carter1985}, 
\[
 B \cap B^{h} \cap B^{h^{-1}}=\langle U_r: r \in \Phi^{+} \cap (\Phi^{+})^{w^{-1}}\cap (\Phi^{+})^{w}   \rangle (B\cap N),
\]
where $\Phi^{+}$ is the set of positive roots and $U_r$ is the root subgroup corresponding to the root $r$.
Therefore, $\vert B \cap B^{h} \cap B^{h^{-1}}\vert=q^m$, where $m$ is the number of roots in $\Phi^{+} \cap (\Phi^{+})^{w^{-1}}\cap (\Phi^{+})^{w} $. 
Since $B\leq P_J$, we have $L_{uvv_1}= P_J \cap P_J^{h} \cap P_J^{h^{-1}}$ and hence $\vert L_{uvv_1}\vert_p\geq q^m$.
To determine $m$, we use the following  {\sc Magma} computation, where \texttt{W} is the Coxeter group defined in the
 previous code snippet: 
\begin{verbatim}
w:=W!(1,16,14,33,22,2)(3,39,34,32,28,43)(4,19,27,15,10,8)(5,18,23)(6,37,7)
(9,46,26,25,40,38)(11,35)(12,21,17)(13,31,30)(20,24)(29,42,47)(36,45,41)(44,48);  
Phiplus:={1..NumPosRoots(W)}; Phiplus meet Phiplus^w meet Phiplus^(w^(-1)); 
\end{verbatim}
The computation gives $m=14$ and hence $\vert L_{uvv_1}\vert_p\geq q^{14}$.
Consequently, $\vert L_v\vert_p \vert L_{uvv_{1}}\vert_p\geq q^{24+14}=q^{38}$, and $\vert L_{uv} \vert_p \vert L_{vv_{1}}\vert_p=q^{17+17}=q^{34}$.
To obtain a contradiction to~\eqref{eq:Gvp},  it now suffices to examine  $\vert H_v/ L_v\vert_p$, $\vert H_{uvv_1}/L_{uvv_1}\vert_p$ and $\vert H_{uv}/ L_{uv}\vert_p$; this can be done using Lemma~\ref{lm:T2at?}.

\begin{lemma}\label{lm:parabolic}
Suppose that Hypothesis~$\ref{hy:1}$ holds.
Then $H_v$ is not a parabolic subgroup. 
\end{lemma}

\begin{proof}
Suppose for a contradiction that $H_v$ is a parabolic subgroup. 
If $L_v$ is not a maximal parabolic subgroup of $L$, then one of the following holds (see e.g.~\cite{C2023}):
\begin{itemize}
\item $L=F_4(q)$ with even $q$, $L_v$ is of type $A_1(q)A_1(q)$ or $C_2(q)$;
\item $L=E_6(q)$, $L_v$ is of type $A_1(q)A_1(q)A_2(q)$ or $D_4(q)$.
\end{itemize}
If $L_v$ is a maximal parabolic subgroup of $L$, then by~\cite[Proposition~10.4.7 and~Theorem~10.4.11]{BCN-book} (see Lemma~\cite[2.13]{YFX2023} for an explanation),
$L$ has a non-self-paired suborbit only in the following cases:
\begin{itemize} 
\item $L=F_4(q)$, $L_v$ is of type $A_1(q)A_2(q)$;
\item $L=E_6(q)$, $L_v$ is of type $A_1(q)A_4(q)$ or $A_1(q)A_2(q)A_2(q)$;
\item $L={}^2E_6(q)$, $L_v$ is of type $A_2(q^2)A_1(q)$ or $ A_1(q^2)A_2(q)$. 
\end{itemize}

 We now treat each candidate case.

Assume that $T=F_4(q)$ and  $ L_v$ is of type $A_1(q)A_2(q)$.
Now $  L_{v} \cong  [q^{20}].(\mathrm{SL}_2(q)\times \mathrm{SL}_3(q)).(q-1)$ (see~\cite[Table~7]{C2023}).
A Magma computation on the Weyl group and roots shows that $L$ has four non‑self‑paired suborbits: two of length $q^7(q^2+q+1)(q+1)$ and two of length $q^{10}(q^2+q+1)(q+1)$. Moreover, if $\vert  L_v\vert/\vert  L_{vv_{1}}\vert=q^7(q^2+q+1)(q+1)$ then $\vert  L_{uvv_{1}}\vert_p \geq q^{14}$, and if $\vert  L_v\vert/\vert  L_{vv_{1}}\vert=q^{10}(q^2+q+1)(q+1)$ then $\vert  L_{uvv_{1}}\vert_p \geq q^{10}$.
Since $\vert H /  L\vert$  divides $f$ (see~\cite[Table~7]{C2023}), Lemma~\ref{lm:T2at?} gives $\vert H_v\vert_p \vert H_{uvv_{1}}\vert_p\geq  \vert L_v\vert_p \vert L_{uvv_{1}}\vert_p\cdot f_p = q^{38}f_p$ or $ q^{34}f_p$,
while $\vert H_{vv_{1}}\vert_p\vert G_{vv_{1}}\vert_p\leq  \vert L_{vv_{1}}\vert_p^2 f_p^2=q^{34}f_p^2 $ or $ q^{28}f_p^2 $, respectively. 
Equation \eqref{eq:Gvp} would then imply $q^{34}f_p^2 > q^{38}f_p$ or $q^{28}f_p^2 > q^{34}f_p$, which is impossible. Hence this case cannot occur.

The same type of computation eliminates all other candidate pairs $(L,L_v)$ except the five listed below:
\begin{enumerate}[\rm (1)]
\item $L=F_4(q)$ with even $q$,  and $L_v$ is a non-maximal parabolic subgroup of type $C_2(q)$, and $v_1$ lies in one of two $L_v$-orbits of length $q^4(q^4+2q^3+2q^2+2q+1)$, or in one of two $L_v$-orbits of length $q^{10}$; 
\item  $L=F_4(q)$ with even $q$, and $L_v$ is a non-maximal parabolic subgroup  of type $ A_1(q) A_1(q)$, and $v_1$ lies in one of two $L_v$-orbits of length $q^2(q+1)^2$, or in one of two $L_v$-orbits of length $q^6(q+1)^2$;

\item $L=E_6(q)$, and $L_v$ is a non-maximal parabolic subgroup of type $ A_1(q) A_1(q) A_2(q)$, and  $v_1$ lies in one of four $L_v$-orbits of length $q^7(q^2+q+1)(q+1)$;

\item $L={}^2E_6(q)$, and $L_v$ is a  maximal parabolic subgroup of type $ A_1(q^2) A_2(q) $, and  $v_1$ lies in one of two $L_v$-orbits of length $q^{11}(q^2+q+1)(q+1)$, or in one of two $L_v$-orbits of length $q^{16}(q^2+q+1)(q+1)$;

\item $L={}^2E_6(q)$, and $L_v$ is a  maximal parabolic subgroup of type $ A_2(q^2) A_1(q) $, and  $v_1$ lies in one of two $L_v$-orbits of length $q^{10}(q^6-1)/(q-1)$, or in one of two $L_v$-orbits of length $q^{14}(q^6-1)/(q-1)$. 
\end{enumerate}

For cases (1)--(3), $H$ contains a graph automorphism \(\gamma\) of order \(2\) that normalises \(L_{v}\) (that is, \(\gamma\in H_{v}\)). 
Moreover, \(\gamma\) maps $\dot{w}$ to $\dot{w}^{-1}$ in $W$, thereby  fusing the double cosets $P_J\dot{w}P_J$ and $P_J\dot{w}^{-1}P_J$ (such a graph automorphism can be located inside $N_{\mathrm{Sym}(\Phi)}(W)$ by a Magma computation; see \cite[Lemma~4.8]{YFX2023} for an example).
Consequently, $\vert (v_1)^{H_v}\vert=2\vert (v_1)^{L_v}\vert$, which contradicts the fact that $ \mathit{\Gamma}$ is an $H$-arc-transitive digraph.  

Assume (4) holds. Then $L_v \cong [q^{29}].(\mathrm{SL}_3(q)\times \mathrm{SL}_2(q)).(q^2-1)/(3,q+1)$ (see~\cite{C2023}). 
Take $r=7$ when $q=4$, and $r\in \mathrm{ppd}(p,3f)$ when $q\neq 4$. 
Then $r $ divides $ \vert H_v\vert$ but does not divide $  \vert H_{vv_{1}}\vert $, contradicting the factorisation $H_v=H_{uv}H_{vv_{1}}$. 

Assume (5) holds. Now $L_{v} \cong  [q^{29}].(\mathrm{PSL}_3(q^2)\times \mathrm{SL}_2(q)).(q-1)$. Take $r=7$ when $q=2$, and $r\in \mathrm{ppd}(p,6f)$ when $q\neq 2$. 
Again $r $ divides $ \vert H_v\vert$ but does not divide $ \vert H_{vv_{1}}\vert $, a contradiction. 

Since every possible configuration leads to a contradiction, $H_v$ cannot be parabolic.
\end{proof}

\subsection{${}^3D_4(q)$}\label{sec:3D4q}
In this subsection, we  assume  Hypothesis~\ref{hy:1} with $L={}^{3}D_{4}(q)$.
The possible subgroups $L_v$ are listed in~\cite{Kleidman1988} (see also~\cite[Table~8.51]{BHRD2013}).
Most of these candidates have already been analysed in~\cite[Lemma~4.7]{YFX2023}. 
Although that \cite[Lemma~4.7]{YFX2023} includes the hypothesis  $\vert V(\mathit{\Gamma})\vert \leq  30,758,154,560$, the part of its proof that deals with the maximal subgroup types relevant to us does not depend on that bound. 
Consequently, we obtain the following corollary.

\begin{corollary}[{\cite[Lemma~4.7]{YFX2023}}]\label{coro:3d4}
Assume Hypothesis~\ref{hy:1} with $L={}^{3}D_{4}(q)$.
Then the only possible structure for $ L_v$ is $\mathrm{C}_{q^{2}\pm q+1}^{2}:\mathrm{SL}_{2}(3)$ with $q\geq 4$.
\end{corollary}
 
Thus, it remains to examine the case $L_v \cong \mathrm{C}_{q^{2}\pm q+1}^{2}:\mathrm{SL}_{2}(3)$.

\begin{lemma}\label{lm:3d4}
Assume Hypothesis~\ref{hy:1} with $L = {}^3D_4(q)$. If  If $L_v\cong \mathrm{C}_{q^{2}\pm q+1}^{2}:\mathrm{SL}_{2}(3)$ with $q\geq 4$, then $s\leq 2$.
\end{lemma}

\begin{proof} 
We apply Lemma~\ref{lm:m2.o} to prove this lemma. 
Note that $\vert H_v / L_v\vert=\vert H/L\vert$ divides $\vert\mathrm{Out}(L)\vert=f$ by Lemma~\ref{lm:T2at?}.

Suppose that $L_{v} \cong \mathrm{C}_{q^{2}+q+1}^2:\mathrm{SL}_{2}(3)$. Let $r=7$ if $q=4$, or $r  \in \mathrm{ppd}(p,3f)$ if $q \neq 4$.
Since $q^3-1=(q-1)(q^2+q+1)$, the prime  $r$ divides $q^2+q+1$ and satisfies $r>f$. 
Moreover, $q(q^2-1)$ is divisible by both $2$ and $3$, so $r \notin \{2,3\}$.
Hence, $|\mathrm{SL}_{2}(3)|_{r}=1$  and $\vert H/L\vert_{r}\leq f_{r}=1$. Lemma~\ref{lm:m2.o} now yields $s \leq 2$.

Suppose that $L_{v} \cong \mathrm{C}_{q^{2}-q+1}^2:\mathrm{SL}_{2}(3)$. 
Choose $r \in \mathrm{ppd}(p,6f)$. Again $|\mathrm{SL}_{2}(3)|_r = 1$ and $|H:L|_r \le f_r = 1$, so Lemma~\ref{lm:m2.o} gives $s \le 2$.
\end{proof}

\subsection{$G_2(q)$}\label{sec:G2q}
In this subsection, we  assume  Hypothesis~\ref{hy:1} with $L=G_{2}(q) $.

 
\begin{corollary}[{\cite[Lemma~4.8]{YFX2023}}]\label{coro:g2q}
Assume Hypothesis~\ref{hy:1} with $L = G_2(q)$. Then $H_v$ satisfies one of the following:
\begin{enumerate}[\rm (i)]
\item $L_v \cong (\mathrm{SL}_{2}(q)\circ \mathrm{SL}_{2}(q)).2$ with odd $q > 19$;
\item $L_v \cong \mathrm{SL}_{2}(q)\times \mathrm{SL}_{2}(q)$ with even $q \ge 32$;
\item $L_v \cong \mathrm{C}_{q^{2}\pm q+1}.6$ with $q = 3^{f} > 9$;
\item $L_v \cong \mathrm{C}_{q\pm 1}^{2}.\mathrm{D}_{12}$ with $q = 3^{f} > 9$.
\end{enumerate}

 
\end{corollary}

\begin{lemma}\label{lm:g2q-1}
Assume Hypothesis~\ref{hy:1} with $L = G_2(q)$. 
If $H_{v}$ satisfies (i) or (ii) of Corollary~\ref{coro:g2q}, then $s\leq 2$.
\end{lemma}

\begin{proof}
Suppose for a contradiction that $s\geq 3$.
Note that $|\mathrm{Out}(L)| = 2f$ if $p = 3$, and $|\mathrm{Out}(L)| = f$ otherwise.
Write $H_v=L_v.[d]$, where $[d]$ denotes a group of order $d$; then  $ \vert H_{v}  \vert= d q^2(q^2-1)^2$ and $d$ divides $\vert\mathrm{Out}(L)\vert=2f$.

Since $s\geq 3$,  Lemma~\ref{lm:HvHuv} gives $
\vert H_{uv}\vert_r \geq \vert H_{ v}\vert_r^{2/3}$  for each prime $r$. 
The group $ (H_v)^{(\infty)}$ has two components $N_1 $ and $N_2$ isomorphic to $\mathrm{SL}_2(q)$.
For each $i\in \{1,2\}$, we have $H_{uv}/(H_{uv}\cap N_i) \cong H_{uv}N_i/N_i \leq H_v/N_i \cong \mathrm{PSL}_2(q).(2,q-1).[d]$, hence
\begin{equation}\label{eq:g2q1}
\vert H_{uv} \cap N_i\vert_r\geq  \frac{
\vert H_{uv}\vert_r}{(dq(q^2-1))_r} \geq \frac{( dq^{2}(q^2-1)^{2})_r^{2/3} }{(dq(q^2-1) )_r}= (d q(q^2-1))^{1/3}. 
\end{equation}

\smallskip

{\noindent\bf Claim.} For each $i\in \{1,2\}$, we have  $H_{uv} \cap N_i \geq N_i$.

Suppose that $q$ is even. Then $f\geq 5$.
Take $r_1 =7 $ if $m=6$, otherwise choose $r_1\in \mathrm{ppd}(2,f)$; take $r_2 \in \mathrm{ppd}(2,2f)$.
Then $r_1 \mid (q-1)$, $r_2 \mid (q+1)$, and $d_{r_1}=d_{r_2}=1$.
From~\eqref{eq:g2q1} we see that $\vert H_{uv} \cap N_i \vert$ is divisible by both $r_1$ and $r_2$. 
Consulting~\cite[Table~10.3]{LPS2000} we conclude $H_{uv} \cap N_i \geq N_i$.

Suppose that $q$ is a Mersenne prime. Thus, $q = p = 2^m - 1$ with $m \ge 5$ (since $q > 19$).
Note that $q-1=2(2^{m-1}-1)$.
Take $r_1 =7 $ if $m=7$, otherwise choose $r_1\in \mathrm{ppd}(2,m-1)$.
Then $d_r=d_p=1$ and $d_2\leq 2$.
By~\eqref{eq:g2q1}, $\vert H_{uv} \cap N_i \vert$ is divisible by both $r_1$ and $p$, and satisfies $\vert H_{uv} \cap N_i \vert_2\geq 2^{m-1}$.
Again~\cite[Table~10.3]{LPS2000} forces $H_{uv} \cap N_i \geq N_i$, as required.
 
Suppose that $q$ is odd and  not a Mersenne prime.
Pick $r_1 \in \mathrm{ppd}(p,2f)$; then $d_{r_1}=1$.
Lemma~\ref{lm:pffp} gives  $q=p^f\geq (f_p)^p>p$.
Therefore, ~\eqref{eq:g2q1} implies $\vert H_{uv} \cap N_i \vert$ is divisible by both $p$ and $r_1$. 
Again by~\cite[Table~10.3]{LPS2000} we have $H_{uv} \cap N_i \geq N_i$.
This completes the proof of the Claim.

\smallskip

By the Claim, $H_{uv}$ contains both $N_1$ and $N_2$, and hence contains $H_{v}^{(\infty)}$.
From $H_{vv_1}=H_{uv}^h$, the same holds for $H_{vv_1}$. 
Consequently, $H_{v}^{(\infty)}$ is normalised by $h$, contradicting Lemma~\ref{lm:Hvnormal}. 
Thus, $s \geq 3$ is impossible, and we must have $s \leq 2$.
\end{proof}

\begin{lemma}\label{lm:g2q-2}
Assume Hypothesis~\ref{hy:1} with $L = G_2(q)$.  
Then $H_v$ does not satisfy (iii) of Corollary~\ref{coro:g2q}.
\end{lemma}

\begin{proof}
Suppose to the contrary that $L_v \cong \mathrm{C}_{q^{2}\pm q+1}.6$ with $q = 3^f > 9$, as Corollary~\ref{coro:g2q}(iii).
Recall that $|H:L|$ divides $|\mathrm{Out}(L)| = 2f$.
Choose $r \in \mathrm{ppd}(p,3f)$ if $L_v\cong \mathrm{C}_{q^{2}+ q+1}.6$, or $r \in \mathrm{ppd}(p,6f)$ if $L_v\cong \mathrm{C}_{q^{2}- q+1}.6$.
In both cases, $6_r=1$ and $\vert H/L\vert_r=1 $. Then Lemma~\ref{lm:m.o} yields $s \le 1$, contradicting the hypothesis $s \ge 2$ of Hypothesis~\ref{hy:1}.
\end{proof}

\begin{lemma}\label{lm:g2q-3}
Assume Hypothesis~\ref{hy:1} with $L = G_2(q)$.
If $H_v$ satisfies (iv) of Corollary~\ref{coro:g2q}, then $s\leq 2$.
\end{lemma}

\begin{proof}
We apply Lemma~\ref{lm:m2.o}. 
If $L_v \cong \mathrm{C}_{q-1}^{2}.\mathrm{D}_{12}$, choose $r \in \mathrm{ppd}(p,3f)$.
If $L_v \cong \mathrm{C}_{q+1}^{2}.\mathrm{D}_{12}$, choose $r \in \mathrm{ppd}(p,6f)$. 
In either situation we have $|\mathrm{D}_{12}|_r = 1$ and $|H:L|_r = 1$, so Lemma~\ref{lm:m2.o} gives $s \le 2$.
\end{proof}

\subsection{${}^2F_4(q) $}\label{sec:2F4q}
In this subsection,  we assume Hypothesis~\ref{hy:1} with $L = {}^2F_4(q) $.
Write $q=2^{2n+1}$ with $n\geq 1$.
According to Malle~\cite{Malle1991} and Craven~\cite[Remark~4.11]{C2023},  $L_v$ satisfies one of the following:
\[
\begin{array}{llp{1cm}ll}
 \text{(i)  $[q^{10}]:({}^2B_{2}(q )\times \mathrm{C}_{q-1})$;}   & \text{(ii) $[q^{11}]:\mathrm{GL}_{2}(q)$;}  \\
 
  \text{(iii) $ \mathrm{SU}_{3}(q):2$;}   & \text{(iv) $\mathrm{PGU}_{3}(q):2$;}\\
  \text{(v) ${}^2B_{2}(q )\wr 2$;} & \text{(vi) $\mathrm{Sp}_{4}(q):2$;}\\
  \text{(vii) ${}^{2} F_{4}(q_{0})'$ with $q=q_{0}^{t}$, for prime $t$;}&\text{(viii) $\mathrm{C}_{q+1}^{2}:\mathrm{GL}_{2}(3)$;}\\
  \text{(ix) $\mathrm{C}_{q\pm \sqrt{2q}+1}^{2}:[96]$;}&\text{(x) $\mathrm{C}_{q^{2}+q+1 \pm \sqrt{2q}(q+1)} :12$;}\\
  \text{(xi) $\mathrm{PGL}_2(13)$ with $q=8$.}
\end{array}
\]

\begin{lemma}\label{lm:2F4q-1}
Aassume Hypothesis~\ref{hy:1} with $L = {}^2F_4(q) $. Then cases (i)--(iv), (vi), (vii) or (xi) can not hold.
\end{lemma}

\begin{proof}
Lemma~\ref{lm:parabolic} excludes parabolic subgroups, so  (i) and (ii) are impossible. 

Suppose (iii) or (iv) holds. Lemma~\ref{lm:qsimple}(a) forces $q = 8$.
Then Lemma~\ref{lm:qsimple}(b.4) together with Table~\ref{tb:exceptionalfacs} shows that $|H_{uv}|_{19} = 19$ while $|H_{vv_1}|_{19} = 1$, a contradiction.

Suppose (vi) holds.  
Now $L_v\cong \mathrm{Sp}_{4}(q):2$.
Let $S$ be the normal subgroup $\mathrm{Sp}_{4}(q)$ of $L_v$.
If $C_{L_v}(S)\cong \mathrm{C}_2$, then $L_v\cong S\times C_{L_v}(S)$ and so $L_v$ is $2$-local, which is impossible by~\cite[Proposition~1.1]{Malle1991}.
Therefore, $C_{L_v}(S)\neq \mathrm{C}_2$ and so it must be the identity subgroup.
Then $L_v$ is an almost simple group with socle $S=\mathrm{Sp}_{4}(q)$.
By Lemma~\ref{lm:qsimple}(b.1), $H_v/\mathrm{Rad}(H_v)$ contains no graph automorphism of $\mathrm{Sp}_{4}(q)$.
However, since $q=2^{2n+1}$, the almost simple group $L_{v} \cong \mathrm{Sp}_{4}(q):2$ certainly contains such a graph automorphism, a contradiction.

Case (vii) is excluded directly by Lemma~\ref{lm:qsimple}(a). 

Suppose (xi) holds. Here $q = 8$ and, according to \cite[Remark~4.11]{C2023}, $H = L$, so $H_v \cong \mathrm{PGL}_2(13)$.  
Lemma~\ref{lm:qsimple}(b.3) then gives $\vert H_{uv}\vert_{13}=1$ while $\vert H_{vv_1}\vert_{13}=13$,  again a contradiction.  
\end{proof}
 
\begin{lemma}\label{lm:2F4q-2}
Aassume Hypothesis~\ref{hy:1} with $L = {}^2F_4(q) $. If case  (v) holds,   i.e., $L_v\cong {}^2B_2(q)\wr 2$, then $s\leq 2$.
\end{lemma}

\begin{proof}
Let $N\cong {}^2B_2(q)^2$ be the base group of $L_v\cong {}^2B_2(q)\wr 2$, and let $N_1, N_2$ be the two normal copies of ${}^2B_2(q)$ of $N$. For $i = 1,2$ denote by $\varphi_i: N \to N_i$ the projection map.
Then $N \unlhd H_v$ and $\vert H_v/N\vert$ divides $2(2n+1)$.
We shall investigate subgroups $H_{uv} \cap N$ and $H_{vv_1} \cap N$.

Let $R$ be the kernel of  $H_v$  acting on $\{N_1, N_2\} $ by conjugation.
We have $H_v/R \cong \mathrm{S}_2$ and $H_v/R= (H_{uv}R/R) (H_{vv_1}R/R)$.
Without loss of generality we may assume $H_{uv}R/R \cong \mathrm{S}_2$; thus $H_{uv}$ acts transitively on $\{N_1, N_2\} $ by conjugation.
Then $\phi_1(N \cap H_{uv})\cong \phi_2(N \cap H_{uv})$. 
Since $N \cap H_{uv} \leq \phi_1(N \cap H_{uv})\times \phi_2(N \cap H_{uv})$, we obtain $\vert\phi_1(N \cap H_{uv})\vert^2$ is divisible by $\vert N\cap H_{uv}\vert$.

Recall $\vert {}^2B_2(q)\vert=q^2(q^2+1)(q-1)$.
Choose primes $r_1 \in \mathrm{ppd}(2,4(2n+1))$ and $r_2\in \mathrm{ppd}(2, 2n+1)$, and write $|{}^2B_2(q)|{r_i} = r_i^{a_i}$ for $i=1,2$.  
Then $\vert H_v\vert_{r_i}=r_i^{2a_i}$ for each $i\in\{1,2\}$ as $\vert H_v/N\vert$ divides $2(2n+1)$.
From $H_v=H_{uv}H_{vv_1}$ we conclude $\vert H_{uv}\vert$ is divisible by both $r_1$ and $r_2$.
Note that $\vert H_{uv}\vert_{r_i}=\vert H_{uv} \cap N\vert_{r_i}$.
It follows from $\vert N\cap H_{uv}\vert$ dividing $\vert\phi_1(N \cap H_{uv})\vert^2$ that $\vert\phi_1(N \cap H_{uv})\vert $  is divisible by both $r_1$ and $r_2$.
Consulting~\cite[Table~10.5]{LPS2000} we conclude  $\phi_1(N\cap H_{uv})=N_1\cong  {}^2B_2(q)$.
Then  $N\cap H_{uv}$ is a subdirect subgroup of $N$.
By Scott's Lemma (see e.g.~\cite[Theorem 4.16]{PS2018}), $N\cap H_{uv}=N\cong {}^2B_2(q)^2$ or $N\cap H_{uv}\cong {}^2B_2(q)$ is a diagonal subgroup of $N $.
If $N\cap H_{uv}=N$, then from $H_{vv_1}=H_{uv}^h$ we have $N \leq H_{vv_1}$ and hence $N^h=N$, contradicting Lemma~\ref{lm:Hvnormal}.
Therefore, $N\cap H_{uv}\cong {}^2B_2(q)$ is a diagonal subgroup of $N $.
Then $\vert H_{uv}\vert_{r_1}=r_1^{a_1}$.
If $s\geq 3$, then from Lemma~\ref{lm:HvHuv} we conclude $\vert H_{uv}\vert_{r_1}\geq r_1^{4a_1/3}>r_1^{a_1}$, a contradiction. 
Therefore, we must have $s\leq 2$.
\end{proof}

\begin{lemma}\label{lm:2F4q-3}
Aassume Hypothesis~\ref{hy:1} with $L = {}^2F_4(q) $. If case  (viii) or case  (ix) holds,  i.e., $L_v\cong \mathrm{C}_{q+1}^2:\mathrm{GL}_2(3)$ or $\mathrm{C}_{q\pm \sqrt{2q}+1}^{2}:[96]$, respectively, then $s\leq 2$.
\end{lemma}

\begin{proof}
We apply Lemma~\ref{lm:m2.o}.  
Note that $\vert H_v / L_v\vert=\vert H/L\vert$ divides $\vert \mathrm{Out}(L)\vert=f=2n+1$ ($n\geq 1$).

For case (viii), take $r=7$ if $q=8$ and $r\in \mathrm{ppd}(2,2(2n+1))$ if $q>8$.
Since $|\mathrm{GL}_2(3)| = 48$, we have $48_r = 1$ and $|H:L|_r \le f_r = 1$. Lemma~\ref{lm:m2.o} yields $s \le 2$. 

For case (viii), set $m=q\pm \sqrt{2q}+1$. 
Observe that $m= 2^{2n+1}-2^{n+1}+1\geq 2^{n+1}+1>2n+1$,  so there exists a prime divisor $r$ of $m$ such that $m_r>(2n+1)_{r}$. 
Clearly, $r \neq 2$. Moreover,   $m=(2^{2n+1}+1)\pm\sqrt{2q}\equiv\pm\sqrt{2q}\equiv\pm 2\pmod{3}$, hence $r \neq 3$. 
Thus, $(96)_r=1$ and  $|H:L|_r \le f_r = 1$.
Lemma~\ref{lm:m2.o} gives $s \le 2$.
\end{proof}

\begin{lemma}\label{lm:2F4q-4}
Aassume Hypothesis~\ref{hy:1} with $L = {}^2F_4(q) $. Then case (x) can not hold.
\end{lemma}

\begin{proof}
Suppose to the contrary that $L_v$ satisfies (x).
Then $L_{v} \cong \mathrm{C}_{q^{2}+q+1 \pm \sqrt{2q}(q+1)} :12$.
Again $|H_v:L_v| = |H:L|$ divides $|\mathrm{Out}(L)| = f = 2n+1$ ($n \ge 1$). 
Set $m=q^{2}+q+1 \pm \sqrt{2q}(q+1) $.
Clearly, $m>2n+1$, so  there exists a prime divisor $r$ of $m$ such that $m_r>(2n+1)_{r}$.
Note that $r\neq 2$ as  $m $ is odd. 
Since $q+1=2^{2n+1}+1\equiv 0\pmod{3}$, we have 
\[m=q^{2}+(q+1)(1\pm\sqrt{2q}) \equiv q^{2}\equiv 2^{4n+2}\equiv1\pmod{3}.\] 
Thus, $r\neq 3$. 
Consequently, $(12)_r = 1$ and $|H:L|_r \le f_r = 1$.
Now Lemma~\ref{lm:m.o} implies $s \le 1$, contradicting the hypothesis $s \ge 2$ of Hypothesis~\ref{hy:1}.
\end{proof}

\subsection{$F_4(q)$}\label{sec:F4q}

In this subsection we assume Hypothesis~\ref{hy:1} with $L = F_4(q)$.
The possible vertex stabilisers $H_v$ are listed in \cite[Theorem~1]{C2023}, which are given in \cite[Tables~1,~7 and~8]{C2023}.
 
\begin{lemma}
Assume Hypothesis~\ref{hy:1} with $L = F_4(q)$. Then $H_v$ cannot belong to \cite[Table~1]{C2023}. 
\end{lemma}
\begin{proof}
If $H_v$ appears in~\cite[Table~1]{C2023}, then $H_v$ is an almost simple group.
However, none of them is isomorphic to $\mathrm{A}_6$, $\mathrm{M}_{12}$, $\mathrm{Sp}_4(2^t)$ or $\mathrm{P\Omega}^{+}_8(q)$, contradicting Lemma~\ref{lm:qsimple}(a).
\end{proof}

We now examine the candidates from \cite[Tables~7 and~8]{C2023}.
By Lemma~\ref{lm:parabolic}, $H_v$ is not parabolic, so we only need to consider the following possibilities:
\begin{enumerate}[\rm (i)]
\item $L_v\cong 3^3:\mathrm{SL}_3(3)$ with $H=L$ and $q=p\geq 5$;

\item   $L_v\cong \mathrm{PGL}_2(q)\times G_2(q) $ with  odd $q\neq 3$;

\item $L_v\cong \mathrm{PGL}_2(q)$ with $p\geq 13$, or $G_2(q)$ with $p=7$; 

\item $L_v\cong F_4(q_0)$, where $q=q_0^t$ with prime $t$, or ${}^2F_4(q_0)$ with $q=q_0^2$;
 
\item $L_v\cong (2,q-1).\Omega_9(q)$,  ${}^3D_4(q).3$, $(2,q-1)^2.\mathrm{P\Omega}_8^{+}(q).\mathrm{S}_3$;

\item $L_v\cong  2.(\mathrm{PSp}_6(q) \times \mathrm{PSL}_2(q)).2$ with odd $q$;

\item $L_{v} \cong \mathrm{Sp}_4(q^2).2$ with even $q$, or  $
(\mathrm{Sp}_4(q)  \times \mathrm{Sp}_4(q)).2$ with even $q$;

\item $L_{v} \cong  (3.q-1).(\mathrm{PSL}_3(q)  \times \mathrm{PSL}_3(q)).(3,q-1).2$, or  $(3,q+1).(\mathrm{PSU}_3(q)  \times \mathrm{PSU}_3(q)).(3,q+1).2$;

\item $L_{v} \cong \mathrm{C}_{q^4-q^2+ 1}.12 $ with even $q>2$.

\item $L_{v} \cong  \mathrm{C}_{q^2+q+1}^2.(3 \times \mathrm{SL}_2(3)) $ with even $q$; 

\item $L_{v} \cong  \mathrm{C}_{q^2+ 1}^2.(\mathrm{SL}_2(3): 4) $ with even $q>2$; (the structure $\mathrm{SL}_2(3) : 4$, correcting $4 \circ \mathrm{GL}_2(3)$ from~\cite{C2023}, see~\cite[Remark 6.1]{BK2025})

\item $L_{v} \cong  \mathrm{C}_{q^2-q+ 1}^2.(3 \times \mathrm{SL}_2(3)) $ with even $q>2$;

\item $L_{v} \cong  \mathrm{C}_{q-1}^4.W(F_4) $ with even $q>4$; 

\item $L_{v} \cong  \mathrm{C}_{q+1}^4.W(F_4) $ with even $q>2$.

\end{enumerate}

We note that those groups in (v)--(xiv) are subgroups of maximal rank.
When analysing such subgroups, we follow the notation introduced in Subsection~\ref{subsec:torus}: $G$ denotes the simple adjoint algebraic group of type $F_4$, $\sigma$ the Steinberg endomorphism, $W$ the Weyl group, $S$ a $\sigma$-stable maximal torus, $\Phi$ the root system, $X$ and $Y$ the root and coroot spaces, and for a $\sigma$-stable closed connected reductive subgroup $E$ containing $S$, $\Delta$ denotes its root subsystem and $W(\Delta)$ its Weyl group. Recall that the Dynkin diagram of $G=F_4$ is as in Subsection~\ref{subsec:torus}. 
For cases (v)--(viii), the stabiliser $L_v\cong  E_{\sigma w}.C_{W_\Delta}(wW(\Delta))$, where $W_\Delta = N_W(\Delta)/W(\Delta)$ and $w\in W$.
For cases (ix)--(xiv), $L_v\cong S_{\sigma w}.C_{W}(w)$ for some $w\in W$. 


%

%
 \begin{lemma}
Assume Hypothesis~\ref{hy:1} with $L = F_4(q)$. 
Then case (i) can not hold.   
\end{lemma}

\begin{proof} 
Suppose to  the contrary that $L_{v} \cong 3^3.\mathrm{SL}_3(3)$, as case (i).
Then  $q$ is odd and $H=L$.
According to~\cite[p.23]{CLSS}, the group  $3^3.\mathrm{SL}_3(3)$ embeds into $\mathrm{PSL}_4(3)$.
Note that $\vert H_v\vert= 2^4 \cdot 3^6 \cdot 13$.
The factorisation $H_{v}=H_{uv}H_{vv_1}$ implies that $\vert H_{uv} \vert$ is divisible by $2^2\cdot 3^3\cdot 13$. 
We construct  $H_v$ as a subgroup of $\mathrm{PSL}_4(3)$ in {\sc Magma}~\cite{Magma} and apply the function \texttt{facsm} defined in Subsection~\ref{subsec:2.1}.
The computation shows that $H_v$ has no such a homogeneous factorisation $H_{v}=H_{uv}H_{vv_1}$ with $\vert H_v\vert/\vert H_{uv}\vert \geq 3$. 
\end{proof}

 \begin{lemma}
Assume Hypothesis~\ref{hy:1} with $L = F_4(q)$.
Then case (ii) can not hold.  
\end{lemma}

\begin{proof} 
Suppose to the contrary that $L_{v} \cong \mathrm{PGL}_2(q)\times G_2(q) $, as case (ii).
Then $q\geq 5$.
Let $N=C_{H_{v}}(\mathrm{PGL}_{2}(q))$.
Then $H_v/N$ is an almost simple group with socle $G_2(q)$, and  $\vert N\vert$ divides $\vert \mathrm{PGL}_2(q)\vert\cdot \vert H/L\vert=fq(q^2-1)$. 
Note that $\vert G_2(q)\vert=q^6(q^6-1)(q^2-1)$.
Choose primes $r_1\in \mathrm{ppd}(p,6f)$ and $r_2\in \mathrm{ppd}(p,3f)$.
Then $\vert N\vert_{r_i}=1$ for each $i\in \{1,2\}$.
The factorisation $H_{v}=H_{uv}H_{vv_1}$ implies that $\vert H_{uv}\vert$ is divisible by both $r_1$ and $r_2$, and that $\vert H_{uv}\vert_p\geq (\vert H_v\vert_p)^{1/2}\geq p^{7f/2}$.

By Lemma~\ref{lm:N1N2corefree}, $H_v/N=(H_{uv}N/N)(H_{vv_1}N/N)$ is a core-free factorisation. 
Since $\vert N\vert_{r_i}=1$ for each $i\in \{1,2\}$,  we see that $\vert H_{uv}N/N \vert$ is divisible by both $r_1$ and $r_2$.
Consulting~\cite[Table~10.5]{LPS2000}, the only possibility for $H_{uv}N/N$ that satisfies this divisibility condition is $H_{uv}N/N =\mathrm{PSL}_2(13)$ with $r_1=13 $ and $r_2=7$.
Now estimate the $p$-part.
Lemma~\ref{lm:pffp} gives $f_p \le p^{f/p}$, so $\vert N\vert_p\leq p^f\cdot f_p\leq p^f\cdot p^{f/p}$.
Consequently,  
\[ \vert H_{uv}N/N\vert_p= \frac{\vert H_{uv} \vert_p}{(H_{uv}\cap N)_p}\geq \frac{p^{7f/2}}{p^{f(1+1/p)}}\geq  p^{2f}=q^2.\]
But $|\mathrm{PSL}_2(13)| = 2^2 \cdot 3 \cdot 7 \cdot 13$ has trivial $p$-part, a contradiction.
\end{proof}
 
\begin{lemma}
Assume Hypothesis~\ref{hy:1} with $L = F_4(q)$. Then $H_v$ cannot satisfy (iii)--(v). 
\end{lemma} 

\begin{proof}
For cases (iii)--(v), $H_v^{(\infty)}$ is quasisimple.
According to Lemma~\ref{lm:qsimple}(a), we only need to  we need only examine the two specific possibilities:  $\mathrm{PGL}_2(q)$ with $p\geq 13$ and $(2,q-1)^2.\mathrm{P\Omega}_8^{+}(q).\mathrm{S}_3$.

Suppose that $L_{v} \cong \mathrm{PGL}_2(q)$. Here $\vert H_v\vert/\vert L_v\vert$ divides $f$.
By Lemma~\ref{lm:qsimple}(b.3), interchanging $H_{uv}$ and $H_{vv_{1}}$ if necessary, we have $H_{uv}\leq \mathrm{D}_{2(q+1)/(2,q-1)}.f$.
Consequently, $\vert H_{uv}\vert_p\leq f_p\leq p^{f/p}$ by Lemma~\ref{lm:pffp}, while $\vert H_{vv_1}\vert_p\geq q=p^f$, contradicting $H_{uv}\cong H_{vv_1}$.

Suppose that $L_{v} \cong (2,q-1)^2.\mathrm{P\Omega}_8^{+}(q).\mathrm{S}_3$.
Then $L_v$ arises from a subsystem $\Delta
$ of type $D_4$, which can be taken as the set of all long roots (see~\cite[Table~B.2]{GD2011}). 
Since $W_{\Delta}=\mathrm{S}_3$ acts as group of symmetries  on the Dynkin diagram of $\Delta$ (by~\cite[Corollary~5]{Carter1978}), $H_v$ induces the full group  $\mathrm{S}_3 $ of graph automorphisms of  $\mathrm{P\Omega}_8^{+}(q) $.
This contradicts Lemma~\ref{lm:qsimple}(b.2).  
\end{proof}

\begin{lemma}
Assume Hypothesis~\ref{hy:1} with $L = F_4(q)$. 
Then  case (vi) can not hold. 
\end{lemma} 

\begin{proof}
Suppose to the contrary that  case (vi) holds, i.e., $L_{v} \cong  2.(\mathrm{PSp}_6(q) \times \mathrm{PSL}_2(q)).2$ with odd $q$.  
Let $N=C_{H_v}(2.\mathrm{PSp}_6(q))  $. Then $H_v/N$ is an almost simple group with socle $\mathrm{PSp}_6(q)$.
Since $\vert H_v\vert/\vert L_v\vert$ divides $f$, we have that $ \vert N \vert$ divides $2f\vert \mathrm{SL}_2(q)\vert=2fq(q^2-1)$.
Consider the factorisation $H_v/N=(H_{uv}N/N) (H_{vv_{1}}N/N)$.
Note that $\vert \mathrm{Sp}_6(q)\vert=q^9(q^6-1)(q^4-1)(q^2-1)$.
Choose primes $r_1 \in \mathrm{ppd}(p,6f)$, $r_2 \in \mathrm{ppd}(p,4f)$ and $r_3 \in \mathrm{ppd}(p,3f)$. 
Then each $r_i$ divides $\vert H_{uv}\vert$ but does not divide  $ \vert N\vert$.
Hence, each $r_i$ divides  $\vert H_{uv}N/N \vert$. 
By Lemma~\ref{lm:N1N2corefree},  both $ H_{uv}N/N  $ and $ H_{vv_{1}}N/N $ are core-free in  $H_v/N$.
However, consulting~\cite[Table~10.1]{LPS2000} we find that $ H_{uv}N/N  $ must contain $\mathrm{PSp}_6(q)$.
Now observe that $H_{v}$ contains a unique  subgroup  isomorphic to $\mathrm{Sp}_{6}(q)$.
Thus, this $\mathrm{Sp}_{6}(q)$ is contained in $ H_{uv}$, and hence in $ H_{uv}$ as $H_{uv}^h=H_{vv_1}$.
Then $\mathrm{Sp}_{6}(q)$ is normalised by $h$, which contradicts Lemma~\ref{lm:Hvnormal}. 
\end{proof}

\begin{lemma}\label{lm:F4q2C2}
Assume Hypothesis~\ref{hy:1} with $L = F_4(q)$. 
If $H_v$ satisfies condition (vii), i.e., $L_v$ is either $\mathrm{Sp}_4(q^2).2$ or $(\mathrm{Sp}_4(q) \times \mathrm{Sp}_4(q)).2$ with $q$ even, then necessarily $L_v \cong (\mathrm{Sp}_4(q) \times \mathrm{Sp}_4(q)).2$ and $s \le 2$.
\end{lemma} 

\begin{proof}
Now $L_v$ arises from a subsystem $\Delta$ of type $2C_2$ described in the example of Subsection~\ref{subsec:torus}.
By~\cite[Table~7]{C2023}, $H_v$ contains a graph automorphism of $F_4(q)$, say $\tau$.  
Using the Magma code from Subsection~\ref{subsec:torus} we execute:
\begin{verbatim} 
NW:=Normaliser(Sym(#Rs),W); NW1:=Stabiliser(NW,SD); 
NW2:=[i: i in NW1|i in W eq false];NW2; 
\end{verbatim}
The computation shows that $\tau$ can be chosen to act as $(\alpha_2,\alpha_3)(-\beta_0,\alpha_0)$ on the simple roots.
Since $\{\alpha_2,\alpha_3\}$ and $\{-\beta_0,\alpha_0\}$ are sets of simple roots for the two $C_2$ factors, $\tau$ induces a graph automorphism on each $\mathrm{Sp}_4(q)$ in the group $(\mathrm{Sp}_4(q) \times \mathrm{Sp}_4(q)).2$.
Similarly, since  $\{(\alpha_2+\alpha_0)/2, (\alpha_3-\beta_0)/2\}$ is a set of simple roots for $\mathrm{Sp}_4(q^2)$, we also have that $\tau$ induces a graph automorphism on $\mathrm{Sp}_4(q^2)$ in the group $\mathrm{Sp}_4(q^2).2$.
Then Lemma~\ref{lm:qsimple}(b.1) implies that the case $L_{v} \cong \mathrm{Sp}_4(q^2).2$ can not hold. Therefore, $L_{v} \cong (\mathrm{Sp}_4(q)  \times \mathrm{Sp}_4(q)).2 $.
 
Let $N_1$ and $N_2$  be the two normal copies of $\mathrm{Sp}_4(q)$ in $H_v^{(\infty)}=\mathrm{Sp}_4(q)  \times \mathrm{Sp}_4(q)$.
Since the outer $2$ in $L_v$ (which corresponds to $W_\Delta$) swaps $N_1$ and $N_2$,  it follows that $H_v$ acts transitively on $\{N_1,N_2\}$ by conjugation.
Let $R$ be the kernel of this action; then $H_v/R \cong \mathrm{S}_2$.
In the factorisation $\mathrm{S}_2\cong H_v/R=(H_{uv}R/R)(H_{vv_1}R/R)$, without loss of generality, we may assume that $H_{uv}R/R\cong \mathrm{S}_2$.
Hence, $H_{uv}$ also acts transitively on ${N_1,N_2}$, and the subgroup $Q:=H_{uv}\cap R$ is of index $2$ in $H_{uv}$.

Suppose for a contradiction that $s\geq 3$.
It follows from Lemma~\ref{lm:HvHuv} that $
\vert H_{uv}\vert_r \geq \vert H_{ v}\vert_r^{2/3}$ for each prime $r$.
Then $\vert Q\vert_r \geq \vert H_{ v}\vert_r^{2/3}\cdot (2_r)^{-1}$.
Write $H_v=L_v.[d]$. Then  $d$ divides $2f$.  
Since $N_i$ is normal in $R$, we have $Q/(Q\cap N_i) \cong QN_i/N_i \leq R/N_i$.
Note that $2\vert R/N_i\vert =\vert H_v\vert/\vert N_i\vert$ and $\vert H_v\vert=\vert N_i\vert^2\cdot 2d$. 
Then, for each prime $r$ and each $i\in \{1,2\}$, the following inequality holds: 
\begin{equation}\label{eq:F4q2}
\vert Q \cap N_i\vert_r\geq  \frac{
\vert Q\vert_r}{ \vert R/N_i\vert_r} 
\geq \frac{\vert H_{ v}\vert_r^{2/3}\cdot (2_r)^{-1} }{ (2_r)^{-1} \vert H_v\vert/ \vert N_i\vert_r}
=\left( \frac{\vert H_{v}\vert_r^2}{(\vert H_v\vert/ \vert N_i\vert)_r^3 }\right)^{1/3} 
=\left( \frac{\vert N_i\vert_r}{2d_r }\right)^{1/3}\geq \left( \frac{\vert N_i\vert_r}{(4 f)_r }\right)^{1/3}.
\end{equation}  

Note that $\mathrm{Sp}_4(q)=q^4(q^4-1)(q^2-1)$ and, by Lemma~\ref{lm:pffp}, $f_2\leq 2^{f/2}$.
Choose primes $r_1 \in \mathrm{ppd}(2,4f)$ and $r_2=7 $ when $q=8$, otherwise $r_2 \in \mathrm{ppd}(2,2f)$.
From~\eqref{eq:F4q2} we see that $\vert Q \cap N_i\vert $ (and so $\vert H_{uv}\cap N_i\vert $) is divisible by both $r_1$ and $r_2$.
Consulting~\cite[Table~10.3]{LPS2000}, the only subgroups of $\mathrm{Sp}_4(q)$ whose orders are divisible by both $r_1$ and $r_2$ are either $Q\cap N_i =N_i\cong \mathrm{Sp}_4(q)$ or $Q\cap N_i \unrhd \mathrm{Sp}_2(q^2)$.
Since $H_{uv}$ acts transitively on $\{N_1,N_2\}$, we have $Q\cap N_1\cong Q\cap N_2$.
If $Q \cap N_i = N_i$ for $i=1,2$, then $H_{v}^{(\infty)}\leq H_{uv}$ and so $(H_{v}^{(\infty)})^h=H_{v}^{(\infty)}$, contradicting  Lemma~\ref{lm:Hvnormal}.
Thus, for $i=1,2$ we have $Q \cap N_i \unrhd \mathrm{Sp}_2(q^2)$.
 
Let $\varphi_i: Q\cap H_v^{(\infty)} \to N_i$ be the projection.
Then $Q\cap H_v^{(\infty)} \leq \varphi_1(Q\cap H_v^{(\infty)}) \times \varphi_2(Q\cap H_v^{(\infty)})$.
Since $Q\cap N_i \unlhd \varphi_i(Q\cap H_v^{(\infty)})$, it follows from~\cite[Table~8.14]{BHRD2013} that $\varphi_i(Q\cap H_v^{(\infty)})\leq \mathrm{Sp}_2(q^2).2$.
Then $|\varphi_i(Q\cap H_v^{(\infty)})|_{r_{2}}=(q^{2}-1)_{r_{2}}$ and so $|Q\cap H_v^{(\infty)}|_{r_{2}}=(q^{2}-1)_{r_{2}}^{2}$. This further implies that \(|H_{uv}|_{r_{2}}=|H_{uv}\cap R|_{r_{2}}=|Q|_{r_{2}}=(q^2-1)_{r_{2}}^{2}\). On the other hand, \(|H_{v}|_{r_{2}}=|\mathrm{Sp}_{4}(q)|_{r_{2}}^{2}=(q^2-1)_{r_{2}}^{4}\). Then, since \(s\geq3\), it follows from Lemma~\ref{lm:HvHuv} that \((q^2-1)^{8}_{r_{2}}=|H_{v}|_{r_{2}}^{2}\leq|H_{uv}|_{r_{2}}^{3}=(q^2-1)_{r_{2}}^{6}\), which is impossible as \((q^2-1)_{r_{2}}>1\). 
Hence $s \ge 3$ cannot occur, and we must have $s \le 2$.
\end{proof}

\begin{lemma}
Assume Hypothesis~\ref{hy:1} with $L = F_4(q)$. 
If $H_v$ satisfies (viii), i.e., $L_{v} \cong   (\mathrm{SL}_3(q) \circ \mathrm{SL}_3(q)).(3,q-1).2$, or  $ (\mathrm{SU}_3(q) \circ \mathrm{SU}_3(q)).(3,q+1).2$, then $s\leq 2$.
\end{lemma} 

\begin{proof}
Suppose that $q=2$. Then $H=L$ or $L.2$.
We construct $\mathrm{Aut}(L)= L.2$ in {\sc Magma}~\cite{Magma} using \texttt{AutomorphismGroupSimpleGroup("F4", 2)}. For both possibilities $H = L$ and $H = L.2$, we then obtain $H_v$ via the command \texttt{MaximalSubgroups}. Direct computation verifies that $H_v$ admits no homogeneous factorisation $H_v = H_{uv} H_{vv_1}$ with $|H_v|/|H_{uv}| \ge 3$. Hence, $s \leq 2$ in this case.

Next, we suppose that $q>2$.
Then $H_v$ is nonsolvable, and  $H_v^{(\infty)}=\mathrm{SL}_3(q) \circ \mathrm{SL}_3(q)$ or $\mathrm{SU}_3(q) \circ \mathrm{SU}_3(q)$.
Let $N_1$ and $N_2$ be the two components of $ H_v^{(\infty)}$ isomorphic to $\mathrm{SL}_3(q) $ or $\mathrm{SU}_3(q)$.

The group $L_v$ arises from a subsystem $\Delta
$ of type $2A_2$. A convenient choice for $\Delta$ is the subsystem whose simple roots are $\{-\alpha_0,\alpha_1,\alpha_3,\alpha_4\}$, where $\alpha_0 = 2\alpha_1+3\alpha_2+4\alpha_3+2\alpha_4$ is the highest root (this subsystem  can be obtained from the extended Dynkin diagram of $F_4$; see~\cite[Theorem B.18]{GD2011}).
In particular, $\{-\alpha_0,\alpha_1 \}$ and $\{\alpha_3,\alpha_4 \}$ are simple roots of two $A_2$ factors, respectively.
For this $\Delta$, the group $W_{\Delta} = N_W(\Delta)/W(\Delta)$ is isomorphic to $\mathrm{C}_2$, generated by the coset of an element $w_0\in W$ that swaps the two $A_2$ factors via $(-\alpha_0,\alpha_1)(\alpha_3,\alpha_4)$.
If we take $w=1$ then $L_{v} \cong   (\mathrm{SL}_3(q) \circ \mathrm{SL}_3(q)).(3,q-1).2$, and if we take $w=w_0$  then $L_{v} \cong    (\mathrm{SU}_3(q) \circ \mathrm{SU}_3(q)).(3,q+1).2$. 
In both cases, the outer `$2$' in $L_v$ corresponds to $W_\Delta \cong \mathrm{C}_2$ and therefore normalises both components $N_1$ and $N_2$.
 
Suppose for a contradiction that $s\geq 3$.
It follows from Lemma~\ref{lm:HvHuv} that $
\vert H_{uv}\vert_r \geq \vert H_{ v}\vert_r^{2/3}$ for each prime $r$. 
Write $H_v=L_v.[d]$. 
Then  $d$ divides $\vert\mathrm{Out}(L)\vert=2f$ by Lemma~\ref{lm:T2at?}.  
Since $N_i$ is normal in $H_v$, we have $H_{uv}/(H_{uv}\cap N_i) \cong H_{uv}N_i/N_i \leq H_v/N_i$.
Note that  $\vert H_v\vert=\vert N_i\vert^2\cdot 2d$. 
Taking $r$-parts we obtain
\begin{equation}\label{eq:F4q1}
\vert H_{uv}\cap N_i\vert_r\geq  \frac{
\vert H_{uv}\vert_r}{ \vert H_v/N_i\vert_r} 
\geq \frac{\vert H_{ v}\vert_r^{2/3} }{ \vert H_v\vert/ \vert N_i\vert_r}
= \left( \frac{\vert N_{i}\vert_r^3 }{ \vert H_v\vert_r } \right)^{1/3} 
=\left( \frac{\vert N_i\vert_r}{(2d)_r }\right)^{1/3}
\geq \left( \frac{\vert N_i\vert_r}{(4f)_r }\right)^{1/3}.
\end{equation}

Suppose that $q\notin \{ 3,5\}$.  Note that $\vert N_i\vert=\vert \mathrm{SL}_3(q)\vert=q^3(q^3-1)(q^2-1)$ or $\vert \mathrm{SU}_3(q)\vert=q^3(q^3+1)(q^2-1)$. 
Choose primes $r_1 \in \mathrm{ppd}(p,6f)$, $r_2 \in \mathrm{ppd}(p,3f)$ and $r_3 \in \mathrm{ppd}(p,2f)$ (note that such prime $r_3$ exists except when $q=p=2^m-1$). 
By~\eqref{eq:F4q1}, $\vert Q\cap N_i \vert$ is divisible by $p$, $r_1$ and $r_3$ (if exists) when $N_i=\mathrm{SU}_3(q)$, and by $p$, $r_2$ and $r_3$  (if exists) when $N_i=\mathrm{SL}_3(q)$. 
From~\cite[Table~10.3]{LPS2000}, we see the following facts:
\begin{itemize}
\item If $q=p=2^m-1$, then $\mathrm{PSU}_3(q)$ has no proper subgroup with order divisible by $p$ and $r_1$, and $\mathrm{PSL}_3(q)$ has no proper subgroup with order divisible by $p$ and $r_2$.
\item If $q$ is not the case that $q=p=2^m-1$, then $\mathrm{PSU}_3(q)$ has no proper subgroup with order divisible by $p$, $r_1$ and $r_3$, and $\mathrm{PSL}_3(q)$ has no proper subgroup with order divisible by $p$, $r_2$ and $r_3$.
\end{itemize} 
Consequently,  $ H_{uv} \cap N_i= N_i $. 
Therefore, $H_v^{(\infty)} \leq H_{uv}$.
From $H_{uv}^h=H_{vv_1}$ we obtain $H_v^{(\infty)}\leq H_{vv_1}$ and that $H_v^{(\infty)}$ is normalised by $h$, leading to a contradiction to Lemma~\ref{lm:Hvnormal}.

Suppose that $q\in \{ 3,5\}$. Now $H=L$ as $\mathrm{Out}(L)=1$, and $f=1$. 
Since $\vert \mathrm{SL}_3(q)\vert_2>4 $ and $\vert \mathrm{SU}_3(q)\vert_2>4 $, it follows from~\eqref{eq:F4q1}  that $\vert H_{uv}\cap N_i\vert$ is divisible by any prime divisor $r$ of $N_i$, for each $i\in \{1,2\}$. 
Direct computation in {\sc Magma} shows the following: if $N_i \cong \mathrm{SL}_3(3)$ or $\mathrm{SL}_3(5)$, then necessarily $H_{uv}\cap N_i = N_i$; if $N_i \cong \mathrm{SU}_3(3)$, then $H_{uv}\cap N_i$ is either $N_i$ or $\mathrm{PSL}_2(7)$; and if $N_i \cong \mathrm{SU}_3(5)$, then $H_{uv}\cap N_i$ is either $N_i$ or $3.\mathrm{A}_7$.

The situation $H_{uv}\cap N_1 = N_1$ and $H_{uv}\cap N_2 = N_2$ leads to the contradiction $H_v^{(\infty)} \le H_{uv}$ as in the case $q\notin\{3,5\}$.
Therefore, $N_1 \cong N_2 \cong \mathrm{SU}_3(q)$, and we may assume without loss of generality that $H_{uv}\cap N_1$ is a proper subgroup: $\mathrm{PSL}_2(7)$ for $q=3$, or $3.\mathrm{A}_7$ for $q=5$.
 
We claim that $H_{uv}\cap N_2=N_2$. Suppose  otherwise; then 
 $H_{uv}\cap N_2\cong \mathrm{PSL}_2(7)$ for $q=3$, or $H_{uv}\cap N_2\cong 3.\mathrm{A}_7$ for $q=5$.
Assume first $q=3$. 
Now $H_{uv}\cap N_1N_2\geq (H_{uv}\cap N_2)\times (H_{uv}\cap N_2)\cong \mathrm{PSL}_2(7)^2$.
Using {\sc Magma} we enumerate all subgroups of $N_1N_2 \cong \mathrm{PSU}_3(3)^2$ whose order is a multiple of $|\mathrm{PSL}_2(7)|^2$. The computation shows that any such subgroup is either conjugate to $\mathrm{PSL}_2(7)^2$ or contains a normal component isomorphic to $\mathrm{PSU}_3(3)$. 
Hence, $H_{uv}\cap N_1N_2 \cong \mathrm{PSL}_2(7)^2$. 
Now compare the $3$-parts. We have $|\mathrm{PSL}_2(7)|_3 = 3$ and $|\mathrm{PSU}_3(3)|_3 = 3^3$. Consequently, $|H_{uv}|_3 =  3^2$, while $|H_v|_3  = 3^6$, contradicting the factorisation $H_v = H_{uv} H_{vv_1}$.
The case $q=5$ is ruled out by an analogous argument.
Therefore, the claim $H_{uv}\cap N_2=N_2$ holds.




Now we consider the structure of $H_{vv_1}= H_{uv}^{h} $.
Note that $H_{vv_1}^{(\infty)}$ is isomorphic to $\mathrm{PSL}_2(7)\times \mathrm{SU}_3(3)  $  (for $q=3$) or $3.\mathrm{A}_7 \circ \mathrm{SU}_3(5) $ (for $q=5$).
Hence, one of $N_1$ or $N_2$ must be contained in $H_{vv_1}$.
If $N_2\leq H_{vv_1}$, then from $N_2\leq H_{vv_1}$ and $H_{uv}^h=H_{vv_1}$ we have $N_2^h=N_2$, contradicting Lemma~\ref{lm:Hvnormal}.
Thus, $N_1\leq H_{vv_1}$ and it follows that $N_2^h=N_1$.

The maximality of $H_v$ in $H$ implies that $H_v=N_H(N_1)=N_H(N_1)=N_H(N_1N_2)$.
According to~\cite[Proposition~2.6.2(c)]{CFSG}, the group $(3,q+1)$ in $H_v$ induces a diagonal automorphism of both $N_1$ and $N_2$.
For each $i\in \{1,2\}$, since $H_v=N_H(N_i)$, we have $ C_H(N_i)= C_{H_v}(N_i)=N_{3-i}$.
From $N_2^h=N_1$ we have
\[
N_{1}^{h}=C_{H}(N_{2})^{h}=C_{H}(N_{2}^{h})=C_{H}(N_{1})=N_{2},
\]
and so $h$ swaps $N_{1}$ and $N_{2}$.
Then $h$ normalises $N_1N_2=H_v^{(\infty)}$, contradicting Lemma~\ref{lm:Hvnormal}.  
\end{proof}

\begin{lemma}
Assume Hypothesis~\ref{hy:1} with $L = F_4(q)$. 
Then case (ix) can not hold.  
\end{lemma}
\begin{proof} 
For case (ix), $L_{v} \cong \mathrm{C}_{q^4-q^2+1}.12$ with even $q>2$.
Set $m=q^4-q^2+1$.
Note that $m$ divides $q^{12}-1 = (q^6-1)(q^2+1)(q^4-q^2+1)$. 
Choose prime $r \in \mathrm{ppd}(2,12f)$.
Thus, $m_r\geq r$, $(12)_r=1$ and $\vert H/L\vert_r=1$. 
Lemma~\ref{lm:m.o} gives $s\leq 1$, contradicting the assumption $s \ge 2$ in Hypothesis~\ref{hy:1}.  
\end{proof}

\begin{lemma}
Assume Hypothesis~\ref{hy:1} with $L = F_4(q)$. 
If $H_v$ satisfies one of (x)--(xii), i.e., $L_{v} \cong  \mathrm{C}_{q^2+q+1}^2.(3 \times \mathrm{SL}_2(3)) $ with even $q$, $\mathrm{C}_{q^2+ 1}^2.(\mathrm{SL}_2(3) : 4) $ with even $q>2$, or $\mathrm{C}_{q^2-q+ 1}^2.(3 \times \mathrm{SL}_2(3)) $ with even $q>2$,  then $s\leq 2$.
\end{lemma} 
\begin{proof} 
For these cases, $L_{v} \cong \mathrm{C}_m^2.\mathcal{O}$ with $\mathcal{O}$ a $\{2,3\}$-group.   
Choose primes $r_{i}\in \mathrm{ppd}(2,fi)$ for $i\in\{3,4,6\}$, and take
\[
r=\begin{cases}
    r_{3}&\text{if $L_{v} \cong \mathrm{C}_{q^2+q+1}^2.(3 \times \mathrm{SL}_2(3))$},\\
    r_{4}&\text{if $L_{v} \cong \mathrm{C}_{q^2+ 1}^2.(\mathrm{SL}_2(3) : 4)$},\\
    r_{6}&\text{if $L_{v} \cong \mathrm{C}_{q^2-q+ 1}^2.(3 \times \mathrm{SL}_2(3))$.}
\end{cases}
\] 
Then $m_r\geq r$, $\vert \mathcal{O}\vert_r=1$ and $\vert H/L\vert_r=1$.
Lemma~\ref{lm:m2.o} gives $s\leq 2$.
\end{proof}

\begin{lemma}\label{lm:F4(q-1)^4}
Assume Hypothesis~\ref{hy:1} with $L = F_4(q)$.  
If $H_v$ satisfies  (xiii), i.e.,  $L_{v} \cong   \mathrm{C}_{q-1}^4.W(F_4)$  with even $q>4$, then $s\leq 2$.     
\end{lemma}

\begin{proof} 
Suppose for a contradiction $s\geq 3$.
By Proposition~\ref{pro:HsMs-1}, $\mathit{\Gamma}$ is $(L,2)$-arc-transitive, and so $L_{v}=L_{uv}L_{vv_1}$.
Let $N =\mathrm{C}_{q-1}^4 \lhd  L_v$, so that $L_v/N \cong W = W(F_4)$.
Consider the factorisation $L_{v}/N=(L_{uv}N/N) (L_{vv_1}N/N)$.

Choose $r\in \mathrm{ppd}(2,f)$ (note that $r\notin\{2,3\}$ as $q=2^f>4$, and $r\geq f+1$).
It follows from $|W|=2^7 \cdot 3^2$ that  $L_v$ has a unique subgroup $R\cong \mathrm{C}_{r}^4$.
Since $\vert L_{uv}\vert $ and $\vert L_{vv_1}\vert $ are divisible by $r$, we have that $R\cap L_{uv}>1$ and $R\cap L_{vv_1}>1$.


We now use the explicit description of maximal tori developed in Subsection~\ref{subsec:torus}.
The stabiliser $L_v$ corresponds to the maximal torus $S_{\sigma w}$ with $w \in W$. 
From $|S_{\sigma w}|=(q-1)^4$, computation in {\sc  Magma} shows that $w = 1$ (the identity element of $W$).
In this case $C_W(w) = W$, and by equations~\eqref{eq:TwF} and~\eqref{eq:xJ}, the group $C_W(w) = W$ acts on the torus $S_{\sigma w} \cong \mathrm{C}_{q-1}^4$ via matrix right multiplication on the exponent vectors (as in~\eqref{eq:xJ}).

In the factorisation $L_{v}/N=(L_{uv}N/N) (L_{vv_1}N/N)$, since $\vert N\vert_2=1$ (as $q$ is even) and $L_{uv}\cong L_{vv_1}$, two factors $L_{uv}N/N$ and $L_{vv_1}N/N$ have isomorphic Sylow $2$-subgroups.  
We begin by computing all factorisations of \( W = L_v / N = (L_{uv}N / N) (L_{vv_1}N / N) \) such that the Sylow \( 2 \)-subgroups of \( L_{uv}N / N \) and \( L_{vv_1}N / N \) are isomorphic and have order at least \( 2^4 \). This is achieved using the function \texttt{facsm}, as defined after Proposition~\ref{pro:factorisation}.  
Next, we construct the natural \( \mathbb{Q}W \)-module \( V \)  (as in Subsection~\ref{subsec:torus}), and restrict it to \( L_{uv}N / N \) and \( L_{vv_1}N / N \).
We then test irreducibility of these restrictions.
Computations in {\sc Magma} show that the restrictions of \( V \) to \( L_{uv}N / N \) and \( L_{vv_1}N / N \) are both irreducible. Since \( r \notin \{2, 3\} \), it follows that both \( L_{uv} \) and \( L_{vv_1} \) act irreducibly on \( R \). Consequently, \( R \) is contained in both \( L_{uv} \) and \( L_{vv_1} \), which implies \( R^h = R \). This contradicts Lemma~\ref{lm:Hvnormal}.
\end{proof}

\begin{lemma}
Assume Hypothesis~\ref{hy:1} with $L = F_4(q)$. 
If $H_v$ satisfies  (xiii), i.e.,  $L_{v} \cong  \mathrm{C}_{q+1}^4.W(F_4)$ with even $q>2$, then $s\leq 2$.    
\end{lemma}
\begin{proof}
The proof follows exactly the same pattern as Lemma~\ref{lm:F4(q-1)^4}, 
with the only change being the choice of the primitive prime divisor:
here we take $r \in \mathrm{ppd}(2,4f)$, which divides $q+1$ rather than $q-1$.
The corresponding Weyl element is the longest element $w_0 \in W$, for which again $C_W(w_0) = W$.
All other steps are identical.
\end{proof}

\subsection{$E_6(q)$ and ${}^2E_6(q)$}\label{sec:E6q2E6q}
 
In this subsection, we assume  Hypothesis~\ref{hy:1} with $L=E_{6}(q)$ or ${}^2E_6(q)$.
According to~\cite[Theorem~1]{C2023}, if $L=E_6(q)$ then $H_v$ appears in~\cite[Tables~2~and~9]{C2023}, and if $L={}^2E_6(q)$ then $H_v$ appears in~\cite[Tables~3~and~10]{C2023}.
 
\begin{lemma}
Assume  Hypothesis~\ref{hy:1} with $L=E_{6}(q)$ or ${}^2E_6(q)$.
Then $H_v$ cannot belong to the families listed in~\cite[Tables~2~and~3]{C2023}.  
\end{lemma}
\begin{proof}
If $H_v$ appears in~\cite[Tables~2~and~3]{C2023}, then $H_v$ is an almost simple group.
By Lemma~\ref{lm:qsimple}(a), there is only one possible case: $H=L=E_6(5)$ and $H_v\cong \mathrm{M}_{12}$. 
Suppose that this case holds. 
By~\cite[Subsection~5.2.2]{C2023}, $E_6(5)$ contains exactly four conjugacy classes of maximal subgroups isomorphic to $\mathrm{M}_{12}$.
We may assume choose $H_v$ to be a representative that, in the $27$-dimensional module $V$ of $E_6(5)$,  fixes a $11$-dimensional subspace $U$ of $V$, and is irreducible on both $U$ and $V/U$.
By Lemma~\ref{lm:qsimple}(b.4), we must have $H_{uv}\cong H_{vv_1} \cong\mathrm{M}_{11}$.
Using {\sc Magma}~\cite{Magma}, we examine the restrictions of all inequivalent irreducible $\mathbb{F}_5\mathrm{M}_{12}$-modules of dimensions $11$ and $16$ to the two conjugacy classes of subgroups isomorphic to $\mathrm{M}_{11}$.
The computation reveals: 
\begin{itemize}
\item Both classes of subgroups isomorphic to  $\mathrm{M}_{11}$  are irreducible on any  $16$-dimensional module of $\mathrm{M}_{12}$; and
\item For an $11$-dimensional irreducible $\mathrm{M}_{12}$-module $V_0$, one class of $\mathrm{M}_{11}$ acts irreducibly on $V_0$, while the other class stabilises a direct sum decomposition $\mathbb{F}_5 \oplus \mathbb{F}_5^{10}$ (i.e., it fixes a $1$-dimensional subspace).
\end{itemize}

Without loss of generality, we may assume that $H_{vv_1}$ belongs to the class that fixes a $1$-dimensional subspace $\langle w \rangle$ of $U$.
Since $H_{vv_1} = H_{uv}^h$ with $h \in L \le \mathrm{GL}(V)$, the subgroup $H_{uv}$ also fixes a $1$-dimensional subspace $\langle e \rangle$ of $V$.
Since $H_{uv}$ is irreducible on $U$, the vector $e$ is not in $U$.
Then $\langle e^{H_{uv}},U \rangle=V$ as $H_{uv}$ is irreducible on $V/U$.
However, since $\langle e\rangle$ is stabilised by $H_{uv}$, we have that $\langle e^{H_{uv}},U \rangle=\langle e\rangle+U$ has dimension $16+1=17$, a contradiction. 
\end{proof}
 

Next, we consider the maximal subgroups listed in \cite[Tables~9 and~10]{C2023} that are not parabolic.
The possibile cases are listed in the following: 

\begin{enumerate}[\rm (i)]
\item  $L=E_6(q)$ or ${}^2E_6(q)$, and $L_{v} \cong  \mathrm{PSL}_3(q)\times G_2(q)$ or $3^{3+3}:\mathrm{SL}_3(3)$;

\item  $L=E_6(q)$ and $L_{v} \cong (2,q-1).(\mathrm{PSL}_2(q) \times \mathrm{PSL}_6(q)).(3,q-1)$;
\item  $L={}^2E_6(q)$ and $L_{v} \cong (2,q-1).(\mathrm{PSL}_2(q) \times \mathrm{PSU}_6(q)).(2,q-1)$;

\item  $L=E_6(q)$ and $L_{v} \cong (3,q+1).(\mathrm{PSL}_3(q^2)\times \mathrm{PSU}_3(q)  ).(3,q+1).2$; 
\item  $L={}^2E_6(q)$ and $L_{v} \cong (3,q-1).(\mathrm{PSL}_3(q^2)\times \mathrm{PSL}_3(q)  ).(3,q-1).2$;

\item $L=E_6(q)$ and $L_{v} \cong (3,q-1).(\mathrm{PSL}_3(q)^3 ).(3,q-1).\mathrm{S}_3$;
\item $L={}^2E_6(q)$ and $L_{v} \cong (3,q+1).(\mathrm{PSU}_3(q)^3 ).(3,q+1).\mathrm{S}_3$;

\item $L=E_6(q)$ and $L_{v} \cong \mathrm{C}_{q^2+q+1}^3/(3,q-1).(3^{1+2}.\mathrm{SL}_2(3))$;

\item $L={}^2E_6(q)$ and $L_{v} \cong \mathrm{C}_{q^2-q+1}^3/(3,q+1).(3^{1+2}.\mathrm{SL}_2(3))$; 

\item $L=E_6(q)$ and $L_{v} \cong \mathrm{C}_{q-1}^6/(3,q-1).W(E_6)$ with  $q>4$;

\item $L={}^2E_6(q)$ and $L_{v} \cong \mathrm{C}_{q+1}^6/(3,q+1).W(E_6)$  with  $q>2$;

\item  $L=E_6(q)$ and $L_{v} \cong {}^2E_6(q_0)$, where $q=q_0^2$;

\item  $L=E_6(q)$ and $L_{v} \cong E_6(q_0).((3,q-1),t)$, where $q=q_0^t$ with prime $t$;

\item $L={}^2E_6(q)$ and $L_{v} \cong {}^2E_6(q_0).((3,q+1),t)$, where $q=q_0^t$ with prime $t$;

\item  $L=E_6(q)$ or ${}^2E_6(q)$, and $L_{v} \cong F_4(q)$, $ \mathrm{PSp}_8(q).2$, $G_2(q)$, $\mathrm{PGL}_3(q).2$ (with $p\geq 5$) or  $\mathrm{PGU}_3(q).2$  (with $p\geq 5$);

\item $L=E_6(q)$ and   $L_{v} \cong \mathrm{PSL}_3(q^3).3$;
\item $L={}^2E_6(q)$ and   $L_{v} \cong \mathrm{PSU}_3(q^3).3$;

\item $L=E_6(q)$ and   $L_{v} \cong ({}^3D_4(q)\times  \frac{q^2+q+1}{(3,q-1)} ).3$;
\item $L={}^2E_6(q)$ and   $L_{v} \cong ({}^3D_4(q)\times \frac{q^2-q+1}{(3,q+1)}).3$;
\item $L=E_6(q)$ and   $L_{v} \cong (4,q-1).(\mathrm{P\Omega}_{10}^{+}(q) \times  \frac{q-1}{(4,q-1)(3,q-1)}).(4,q-1)$;
\item $L={}^2E_6(q)$ and   $L_{v} \cong (4,q+1).(\mathrm{P\Omega}_{10}^{-}(q) \times  \frac{q-1}{(4,q+1)(3,q+1)}).(4,q+1)$;

\item $L=E_6(q)$ and   $L_{v} \cong  (2,q-1)^2.(\mathrm{P\Omega}_8^{+}(q)\times (\frac{q-1}{(2,q-1)})^2/(3,q-1)).(2,q-1)^2.\mathrm{S}_3$;
\item $L={}^2E_6(q)$ and   $L_{v} \cong  (2,q-1)^2.(\mathrm{P\Omega}_8^{+}(q)\times (\frac{q+1}{(2,q-1)})^2/(3,q+1)).(2,q-1)^2.\mathrm{S}_3$.

\end{enumerate}

We note that the maximal subgroups appearing in cases~(ii)--(xi) and  (xvi)-(xxiii)  are subgroups of maximal rank.
For these subgroups we shall use the notation set up in Subsection~\ref{subsec:torus}.
Let the Dynkin diagram of $G = E_6$ be as follows, where $\Pi=\{\alpha_1,\dots,\alpha_6\}$ is the set of the simple roots. 

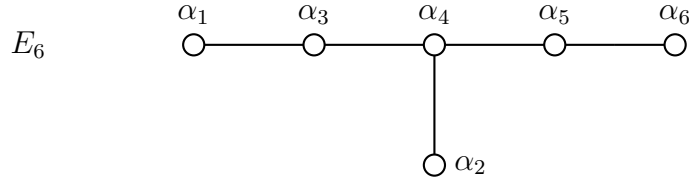
\begin{figure}[h]
\begin{center}
\begin{tikzpicture}[scale=0.4]
\draw [thick] (2,-2.5) -- (-2,-2.5) ;
\draw [thick] (2,-2.5) -- (2,-6.5) ;  
\draw [thick] (-2,-2.5) -- (-6,-2.5) ;
\draw [thick] (2,-2.5) -- (6,-2.5) ;
\draw [thick] (6,-2.5) -- (8,-2.5) ;
\draw [thick] (8,-2.5) -- (10,-2.5) ;


\filldraw[thick,fill=white] (2,-2.5) circle (10pt);
\node at (2,-1.5) {\normalsize $\alpha_4$}; 
\filldraw[thick,fill=white] (-2,-2.5) circle (10pt);
\node at (-2,-1.5) {\normalsize $\alpha_3$}; 
\filldraw[thick,fill=white] (-6,-2.5) circle (10pt);
\node at (-6,-1.5) {\normalsize $\alpha_1$};
\filldraw[thick,fill=white] (6,-2.5) circle (10pt);
\node at (6,-1.5) {\normalsize $\alpha_5$};
\filldraw[thick,fill=white] (2,-6.5) circle (10pt);
\node at (3.2,-6.5) {\normalsize $\alpha_2$};
\filldraw[thick,fill=white] (10,-2.5) circle (10pt);
\node at (10,-1.5) {\normalsize $\alpha_6$}; 
\node at (-11.5,-2.5) {\normalsize $E_{6}$}; 
\end{tikzpicture}
\end{center}
 \end{figure}

We remark on the structure of the Weyl group $W$ of $G = E_6$.
A computation in {\sc Magma}~\cite{Magma} shows that $N_{\mathrm{Sym}(\Phi)}(W) = W{:}\langle \tau \rangle$, where $\tau$ acting as $(\alpha_1,\alpha_6)(\alpha_3,\alpha_5)$ on  $\Pi$.
Let $w_0$ be the longest element of $W$. Then $w_0$ acts as 
\[
(\alpha_1,-\alpha_6)(-\alpha_1,\alpha_6)(\alpha_3,-\alpha_5)(-\alpha_3,\alpha_5)(\alpha_2,-\alpha_2)(\alpha_4,-\alpha_4)
\]
on  $\Pi \cup -\Pi$.
Consequently, the product $w_0\tau$ acts as $-1$ on the entire root system $\Phi$. 


\begin{lemma}
Assume  Hypothesis~\ref{hy:1} with $L=E_{6}(q)$ or ${}^2E_6(q)$.
Then cases (xii)--(xxiii) can not hold. 
\end{lemma}
\begin{proof}
In all these cases, the group $H_v^{(\infty)}$ is quasisimple.
By Lemma~\ref{lm:qsimple}(a), the unique nonabelian simple composition factor of $H_v$ must belong to the list~\eqref{eq:groupsqs}, namely one of
\[
\mathrm{A}_6,\ \mathrm{M}_{12},\ \mathrm{Sp}_4(2^t),\ \mathrm{P\Omega}^{+}_8(q),\ 
\mathrm{PSL}_2(q),\ \mathrm{PSL}_3(3),\ \mathrm{PSL}_3(4),\ \mathrm{PSL}_3(8),\ 
\mathrm{PSU}_3(8),\ \mathrm{PSU}_4(2).
\] 

Inspecting the structures of $L_v$, we see that only cases (xxii) and (xxiii) could potentially have $\mathrm{P\Omega}^{+}_8(q)$ as a composition factor.  
Now $L_v$ arises from a subsystem $\Delta$ of type $D_4$.
Such  subsystem $\Delta$ can be taken to be the one with simple roots $\{\alpha_2,\alpha_3,\alpha_4,\alpha_5\}$ ($E_6$ has a maximal closed subsystem $D_5$, and $D_5$ has a maximal closed subsystem $D_4$). 
Computation in {\sc Magma}~\cite{Magma} shows that $W_{\Delta}\cong \mathrm{S}_3$, and $W_{\Delta}$ is generated by two involutions acting as $(\alpha_2,\alpha_3)$ and $(\alpha_2,\alpha_5)$.
Consequently, in the almost simple group $H_v/\mathrm{Rad}(H_v)$ (whose socle is $\mathrm{P\Omega}_8^{+}(q)$), the outer automorphism group contains the full group $\mathrm{S}_3$ of graph automorphisms, contradicting  Lemma~\ref{lm:qsimple}(b.2).
Thus cases (xxii) and (xxiii) are also impossible. 
\end{proof}

It remains to examine cases (i)--(xi), which consist of subgroups of maximal rank and one exception case (i).  
We will treat them in a series of lemmas.

\begin{lemma}
Assume  Hypothesis~\ref{hy:1} with $L=E_{6}(q)$ or ${}^2E_6(q)$.
Then case (i) can not hold. 
\end{lemma} 

\begin{proof} 
Case (i) comprises two possibilities: $L_v \cong \mathrm{PSL}_3(q) \times G_2(q)$ or $L_v \cong 3^{3+3}{:}\mathrm{SL}_3(3)$.

Suppose that $L_{v} \cong \mathrm{PSL}_3(q)\times G_2(q)$.
Let $N=C_{H_{v}}(G_2(q))$.
Then  $H_v/N$ is an almost simple group with socle $G_2(q)$. 
By Lemma~\ref{lm:N1N2corefree}, in the factorisation $H_v/N=(H_{uv}N/N)(H_{vv_1}N/N)$, both factors  are core-free.
Consulting~\cite[Table~5]{LPS1990}, we find that such core-free factorisation exists  either when $p = 3$ or when $q = 4$. 
Moreover, in each such factorisation, only one of the factors has order divisible by a primitive prime divisor $r\in \mathrm{ppd}(p,6f)$.  
However, since $\vert N\vert$ divides $6f\vert\mathrm{PSL}_3(q)\vert$, we obtain from $H_v=H_{uv}H_{vv_1}$ that both $\vert H_{uv}N/N\vert $ and $\vert H_{vv_1}N/N \vert $ are divisible by $r$, a contradiction.

Suppose that $L_{v} \cong 3^{3+3}.\mathrm{SL}_3(3)$.
Now $H=L$ or $H=L.2$ by~\cite[Tables~9~and~10]{C2023}.
According to Lemma~\ref{lm:T2at?}, $\vert L_{uv}L_{vv_1}\vert= \vert L_v\vert$ or $ \vert L_v\vert/2$ and $L_{uv}$ is not conjugate to $L_{vv_1}$ in $L$.
Since $\vert L_v\vert= 2^4 \cdot 3^9 \cdot 13$, we have that $\vert L_{uv}\vert=\vert L_{vv_1}\vert $ is divisible by $2^2\cdot 3^5\cdot 13$.
By~\cite[p.32]{CLSS}, the group $3^{3+3}.\mathrm{SL}_3(3)$ embeds into $\Omega_7(3)$.
In {\sc Magma}~\cite{Magma}, we construct  $L_v$ in $\Omega_7(3)$, and further computation shows that $L_v$ has no such subgroups $L_{uv}$ and $L_{vv_1}$ satisfying: $|L_{uv}| = |L_{vv_1}| \geq 2^2 \cdot 3^5 \cdot 13$, and $L_{uv}$ and $L_{vv_1}$ are not conjugate in $L_v$, and $|L_v|/|L_{uv}| \geq 3$.  Hence, this  case is also impossible.
\end{proof}

\begin{lemma}
Assume  Hypothesis~\ref{hy:1} with $L=E_{6}(q)$ or ${}^2E_6(q)$.
Then cases (ii)--(v) can not hold. 
\end{lemma}
\begin{proof} 

Suppose for a contradiction that $H_v$ satisfies one of these cases.
Then $H_v$ has a normal subgroup $N$ such that $H_v/N$ is an almost simple group with socle $\mathrm{PSL}_6(q)$ (for (ii)), $\mathrm{PSU}_6(q)$ (for (iii)), $\mathrm{PSL}_3(q^2)$ (for (vi) and (v)),  respectively.
Moreover, $|N|$ divides $36f|\mathrm{PSL}_2(q)|$ in cases (A.6) and (A.7), and divides $36f|\mathrm{PSU}_3(q)|$ or $36f|\mathrm{PSL}_3(q)|$ in cases (A.8) and (A.9), respectively (because $|H:L|$ divides $6f$ by Lemma~\ref{lm:T2at?}, and the structure of $L_v$ includes an additional factor of order at most $6$). 

By Lemma~\ref{lm:N1N2corefree}, in the factorisation $H_v/N=(H_{uv}N/N)(H_{vv_1}N/N)$, both factors are core-free. We now examine the possible factorisations of the almost simple groups with socle $\mathrm{PSL}_6(q)$, $\mathrm{PSU}_6(q)$ and $\mathrm{PSL}_3(q^2)$, using~\cite[Tables~1 and~3]{LPS1990}.
In each case we can choose a suitable prime $r$ with the property that in any core‑free factorisation  
$H_v/N = (H_{uv}N/N)(H_{vv_1}N/N)$, only one of the two factors has order divisible by $r$.  
The choice of the prime $r$ depends on the structure of $\mathrm{Soc}(H_v/N) $:  in case (ii), where \( \mathrm{Soc}(H_v/N) \cong \mathrm{PSL}_6(q) \), we set \( r = 31 \) when \( q = 2 \) and otherwise take \( r \in \mathrm{ppd}(p,6f) \); in case (iii), where \( \mathrm{Soc}(H_v/N) \cong \mathrm{PSU}_6(q) \), we set \( r = 7 \) if \( q = 2 \) and otherwise take \( r \in \mathrm{ppd}(p,10f) \); and in cases (iv) and (v), where \( \mathrm{Soc}(H_v/N) \cong \mathrm{PSL}_3(q^2) \), we take \( r \in \mathrm{ppd}(p,4f) \). 

In each of the four cases, the structure of $L_v$ ensures that $r$ does not divide $|N|$.
Consequently, $r$ divides $|H_{uv}|$ (resp. $|H_{vv_1}|$) if and only if it divides $|H_{uv}N/N|$ (resp. $|H_{vv_1}N/N|$). 
Since $H_v = H_{uv}H_{vv_1}$, the prime $r$ must divide both $|H_{uv}|$ and $|H_{vv_1}|$, hence both $|H_{uv}N/N|$ and $|H_{vv_1}N/N|$.
This contradicts the fact that in each of the above factorisations, only one factor can be divisible by $r$. 
\end{proof}

\begin{lemma}
Assume  Hypothesis~\ref{hy:1} with $L=E_{6}(q)$ or ${}^2E_6(q)$.
Then cases (vi)--(vii) can not hold. 
\end{lemma}
\begin{proof} 
In cases (vi) and (vii) we have respectively
\[
L_v \cong (3,q-1).(\mathrm{PSL}_3(q)^3).(3,q-1).\mathrm{S}_3 \quad\text{or}\quad 
(3,q+1).(\mathrm{PSU}_3(q)^3).(3,q+1).\mathrm{S}_3.
\]
These groups arise from a maximal closed subsystem $\Delta$ of type $3A_2$, whose simple roots can be taken as 
$\{\alpha_1,\alpha_3,\alpha_2,-\alpha_0,\alpha_5,\alpha_6\}$, where $\alpha_0=\alpha_1+2\alpha_2+2\alpha_3+3\alpha_4+2\alpha_5+\alpha_6$ is the highest root of $\Phi$.
The three $A_2$ subsystems have simple roots $\{\alpha_1,\alpha_3\}$, $\{\alpha_2,-\alpha_0\}$ and $\{\alpha_5,\alpha_6\}$, respectively.
The group $W_\Delta = N_W(\Delta)/W(\Delta)$ is isomorphic to $\mathrm{S}_3$ and permutes these three $A_2$ subsystems; 
explicit generators are given in~\cite[Example~1.4]{LSS1992}.  
Consequently, the outer $\mathrm{S}_3$ in $L_v$ permutes the three components $\mathrm{SL}_3(q)$ (or $\mathrm{SU}_3(q)$).
 
\medskip
\noindent\textbf{Case 1.} Suppose that $L_v \cong (3,q-1).(\mathrm{PSL}_3(q)^3).(3,q-1).\mathrm{S}_3$.
  
Since $|H:L|$ divides $6f$, we have $|H_v|/|L_v|$ dividing $6f$.
Choose primes $r_1\in \mathrm{ppd}(p,3f)$ and $r_2\in \mathrm{ppd}(p,2f)$, and write $\vert \mathrm{PSL}_3(q)\vert_{r_j}=r_j^{a_j}$ for  $j=1,2$. 
Then $\vert H_v\vert_{r_j}=r_j^{3a_j}$ and $\vert H_v\vert_p\geq q^9$.
From $H_{v}=H_{uv}H_{uv_1}$, we have  
\begin{equation}\label{eq:E6-A2A2A2}
\vert H_{uv}\vert_{r_j}\geq r_j^{3a_j/2}  \text{ for each }  j\in \{1,2\},\  \vert H_{uv}\vert_{p}\geq p^{9f/2}.
\end{equation}

Write $N = H_v^{(\infty)} = (3,q-1).(\mathrm{PSL}_3(q))^3$ (a central product of three $\mathrm{SL}_3(q)$, see~\cite[Example~1.4]{LSS1992}), 
and let $N_1,N_2,N_3$ be its three components.
Let $R$ be the kernel of $H_v$ acting on $\{N_1,N_2,N_3\}$ by conjugation.
Then $H_v/R\cong \mathrm{S}_3$. 
In the factorisation $ H_v/R=(H_{uv}R/R)(H_{vv_1}R/R)$, at least one factor contains $\mathrm{A}_3$.
We may assume that $H_{uv}R/R \geq \mathrm{A}_3$, and so $H_{uv}$ acts transitively on $\{N_1,N_2,N_3\}$.

Let $\varphi_i: H_{uv}\cap N \to N_i$, $i=1,2,3$, be the projection. 
Then all $\varphi_i(H_{uv}\cap N)$ are isomorphic, and  $\vert H_{uv}\cap N \vert $ divides $ \vert\varphi_1(H_{uv}\cap N)\vert^3 $.
It follows from~\eqref{eq:E6-A2A2A2} that $ \vert\varphi_1(H_{uv}\cap N)\vert $ is divisible by $r_1$, $r_2$ and $p$.
Consulting~\cite[Table~10.3]{LPS2000} we have $\varphi_1(H_{uv}\cap N)\cong N_1$. 
Let $Z$ be the centre $(3,q-1)$ of $N$.
Then $(H_{uv}\cap N)/Z $ is a subdirect product of $\mathrm{PSL}_3(q)^3$. 
Since $H_{uv}$ acts primitively on $\{N_1,N_2,N_3\}$, Scott's Lemma (see e.g.~\cite[Theorem~4.16]{PS2018}) implies that either
$(H_{uv}\cap N)/Z =\mathrm{PSL}_3(q)^3$ or $(H_{uv}\cap N)/Z \cong \mathrm{PSL}_3(q)$ is a  diagonal subgroup.
The first possibility yields a contradiction to Lemma~\ref{lm:Hvnormal}.
The second possibility would give  $|H_{uv}\cap N|_p = |\mathrm{SL}_3(q)|_p = q^3$, contradicting~\eqref{eq:E6-A2A2A2}.

\medskip
\noindent\textbf{Case 2.} Suppose that $L_{v} \cong (3,q+1).(\mathrm{PSU}_3(q)^3 ).(3,q+1).\mathrm{S}_3$. 

The argument is completely analogous to Case~1, using the fact that $\mathrm{PSU}_3(q)$ is simple for $q>2$ and that the same divisibility constraints hold. 
Hence, we only need to show that the subcase $q=2$ can not hold.

Let $q=2$. Then $L_{v} \cong 3.(\mathrm{PSU}_3(2)^3).3.\mathrm{S}_3$ and $\vert L_v\vert= 2^{10} \cdot 3^9$. 
The outer automorphism group of $L$ can involve a field automorphism $\phi$ (order $2$) and/or a diagonal automorphism $\delta$ (order $3$).
Both $\delta$ and $\phi$ act componentwise on $\mathrm{PSU}_3(2)^3$ (see~\cite[Proposition~2.6.2(c)]{CFSG} and~\cite[Lemma~1.3]{LSS1992}).
Consequently, $H_v/Z \le \mathrm{P\Gamma U}_3(2) \wr \mathrm{S}_3$, where $Z$ denotes the centre of $L_v$.


Let $J = \mathrm{P\Gamma U}_3(2) \wr \mathrm{S}_3$ with base group $R = \mathrm{P\Gamma U}_3(2)^3$.
By enumerating subgroups of $J$ of order $2^{10}3^9$ that map onto $J/R \cong \mathrm{S}_3$, we find the possible candidates for $H_v/Z$ corresponding to each possible $H$.
For every possible $H_v/Z$, computation shows that $\mathbf{O}_3(H_v/Z) \cong 3^6$ (the largest normal $3$-subgroup).
Hence $\mathbf{O}_3(H_v)$ coincides with the normal subgroup $3.3^6$ of $3.(\mathrm{PSU}_3(2)^3) = 3.(3^2{:}\mathrm{Q}_8)$.

In the factorisation $H_v/Z = (H_{uv}Z/Z)(H_{vv_1}Z/Z)$, we may assume that $\vert H_{uv}Z/Z\vert_3 \geq \vert H_{vv_1}Z/Z\vert_3$.
For every possible $H_v/Z$, we search such factorisation satisfying: 
\begin{itemize}
\item Both $\vert H_{uv}Z/Z \vert$ and $\vert H_{vv_1}Z/Z \vert$ are divisible by $ 2^{5}\cdot 3^5$ (or $2^6\cdot 3^5$ when $H=L.2$ or $L.3.2$).
\item The Sylow $2$-subgroups of $ L_{uv}Z/Z$ and $L_{vv_1}Z/Z $ are isomorphic.
\item $\vert L_{uv}Z/Z \vert_3 = \vert L_{vv_1}Z/Z \vert_3$ or $3\vert L_{vv_1}Z/Z \vert_3$, and $\vert \mathbf{O}_3( L_{uv}Z/Z)\vert=\vert \mathbf{O}_3( L_{vv_1}Z/Z)\vert$ or $3\vert \mathbf{O}_3( L_{vv_1}Z/Z)\vert $.
\end{itemize} 
The computation reveals that in every admissible factorisation we have 
\[ \mathbf{O}_3( L_{uv}Z/Z)=\mathbf{O}_3( L_{vv_1}Z/Z)=\mathbf{O}_3(H_v/Z)\cong \mathrm{C}_3^6.\] 

Thus $\mathbf{O}_3(H_{uv})Z = \mathbf{O}_3(H_v)$. 
If $\mathbf{O}_3(H_{uv})\cap Z=1$, then $\mathbf{O}_3(H_{uv})Z=\mathbf{O}_3(H_{uv})\times Z=\mathbf{O}_3(H_v) $,  implying $\mathbf{O}_3(H_{uv}) \cong \mathrm{C}_3^6$ and hence $\mathbf{O}_3(H_v) \cong \mathrm{C}_3^7$.
However, the Sylow $3$-subgroup of $\mathrm{SU}_3(2)$ is not abelian, a contradiction.   
Therefore, we must have $Z \leq \mathbf{O}_3(H_{uv})$, which implies $\mathbf{O}_3(H_{uv})=\mathbf{O}_3(H_{v})$.
Similarly, from $\mathbf{O}_3( L_{vv_1}Z/Z)=\mathbf{O}_3(H_v/Z)$ we can obtain $\mathbf{O}_3(H_{vv_1})=\mathbf{O}_3(H_{v})$.

Since $H_{uv}^h=H_{vv_1}$, the group $\mathbf{O}_3(H_{uv})^h$ is a normal $3$-subgroup of $H_{vv_1}$.
Consequently, $\mathbf{O}_3(H_{v})^h=\mathbf{O}_3(H_{uv})^h=\mathbf{O}_3(H_{vv_1})=\mathbf{O}_3(H_{v})$, which contradicts Lemma~\ref{lm:Hvnormal}. 
Therefore, the case $q=2$ also cannot hold. It completes the proof.
\end{proof}




\begin{lemma}\label{lm:E6q(q^2+q+1)^3}
Assume Hypothesis~$\ref{hy:1}$ holds with $L=E_6(q)$.
If  case (viii) holds, i.e., $L_{v} \cong \mathrm{C}_{q^2+q+1}^3/(3,q-1).(3^{1+2}.\mathrm{SL}_2(3))$,  then $s\leq 2$. 
\end{lemma}
\begin{proof}
Here, the stabiliser $L_v$ corresponds to a maximal torus $S_{\sigma w}$ with $w \in W$.
We have $G_{\sigma w} = \mathrm{Inndiag}(E_6(q)) = E_6(q).(3,q-1)$, $S_{\sigma w} \cong \mathrm{C}_{q^2+q+1}^3$, and $C_W(w) \cong 3^{1+2}.\mathrm{SL}_2(3)$.
The element $w$ can be located explicitly, and a {\sc Magma} computation confirms that $C_W(w) \cong \mathrm{GU}_3(2)$.

Suppose for a contradiction that $s\geq 3$. 
By Proposition~\ref{pro:HsMs-1},  $\mathit{\Gamma}$ is $(L,2)$-arc-transitive, and so $L_{v} \cong L_{uv}L_{vv_1}$.  
Let $N$ be the normal subgroup $\mathrm{C}_{q^2+q+1}^3/(3,q-1)$, that is, $N=L \cap S_{\sigma w}$.
Then $L_v/N\cong C_{W}(w)\cong  \mathrm{GU}_3(2)$.
Consider the factorisation  $L_v/N=(L_{uv}N/N)(L_{vv_1}N/N)$.  
Since $|N|$ is odd and $L_{uv} \cong L_{vv_1}$, the two factors have isomorphic Sylow $2$-subgroups.
Using {\sc Magma} we enumerate all factorisations of $\mathrm{GU}_3(2)$ with this property.
Among such factorisations found, one factor is always either $\mathrm{SU}_3(2)$ or $\mathrm{GU}_3(2)$ itself.


Assume that $L_{uv}N/N= \mathrm{SU}_3(2)$ and that $L_{vv_1}N/N$ is a proper subgroup of $L_v/N$.  The computation shows that $L_{vv_1}N/N$ is then $\mathrm{SL}_2(3)$ or $\mathrm{C}_3 \times \mathrm{SL}_2(3)$, whose derived subgroup is $\mathrm{Q}_8$.
Now $(L_{vv_1}N/N)' \cong L_{vv_1}'N/N$, 
and because $|N|$ is odd, $|L_{vv_1}'|_2 = |L_{vv_1}'N/N|_2 = 8$.
Since $L_{uv} \cong L_{vv_1}$, we also have $|L_{uv}'|_2 = 8$.
But $L_{uv}N/N = \mathrm{SU}_3(2)$ has derived subgroup $(\mathrm{SU}_3(2))' \cong 3^{1+2}:2$, whose $2$-part is $2$, a contradiction.

Therefore, \(L_{v}/N=(L_{uv}N/N)(L_{vv_{1}}N/N)\) is not a proper factorisation. Hence, without loss of generality, we may assume that $L_{uv}N/N= \mathrm{GU}_3(2)$.
Note that $q^3-1=(q-1)(q^2+q+1)$. 
Take $r \in \mathrm{ppd}(p,3f)$  and write $(q^2+q+1)_r=r^a$.
Then $\vert L_v\vert_r=r^{3a}$, and it follows from $L_{v} \cong L_{uv}L_{vv_1}$ that $\vert L_{uv}\vert_r\geq r^{3a/2}$. 
Let  $R \cong \mathrm{C}_r^3$ be the unique Sylow $r$-subgroup of $N$ (hence characteristic in $L_v$). 
The group $\mathrm{GU}_3(2) = L_{uv}N/N$ acts irreducibly on $R$ (see~\cite[Lemma~4.6(i)]{LSS1992}).
Since $|L_{uv}|_r \ge r^{3a/2} > r^a$, the intersection $L_{uv} \cap R$ is non‑trivial.
Irreducibility of the action forces $R \le L_{uv}$.
By symmetry, $R \le L_{vv_1}$ as well.
Thus $R$ is contained in both $L_{uv}$ and $L_{vv_1}$, and therefore $R^h = R$, contradicting Lemma \ref{lm:Hvnormal}.  
\end{proof}

\begin{lemma}\label{lm:2E6q(q^2-q+1)^3}
Assume Hypothesis~$\ref{hy:1}$ holds with $L={}^2E_6(q)$.
If  case (ix) holds, i,e, $L_{v} \cong \mathrm{C}_{q^2+q+1}^3/(3,q+1).(3^{1+2}.\mathrm{SL}_2(3))$,  then $s\leq 2$. 
\end{lemma}
\begin{proof}
The argument is completely analogous to that of Lemma~\ref{lm:E6q(q^2+q+1)^3}.  
The only differences are that the torus is now $\mathrm{C}_{q^2-q+1}^3$ and the relevant primitive prime divisor lies in $\mathrm{ppd}(p,6f)$.  
All other parts of the proof, including the structure of $C_W(w) \cong \mathrm{GU}_3(2)$ and the subsequent factorisation analysis, remain identical.
Therefore, $s \leq 2$.
\end{proof}

 \begin{lemma}\label{lm:E6(q-1)6}
Assume Hypothesis~$\ref{hy:1}$ holds with  $L=E_6(q)$.
If case (x) holds, i.e., $L_{v} \cong \mathrm{C}_{q-1}^6/(3.q-1).W(E_6)$  with $q>4$, then $s\leq 2$. 
\end{lemma}    
\begin{proof}   
Suppose for a contradiction that $s \geq 3$.
By Proposition~\ref{pro:HsMs-1}, $\mathit{\Gamma}$ is $(L,2)$-arc-transitive, so $L_v = L_{uv}L_{vv_1}$.
Let $N = \mathrm{C}_{q-1}^6/(3,q-1) \lhd L_v$, so that $L_v/N \cong W = W(E_6) \cong \mathrm{P\Gamma U}_4(2)$ (note $|W| = 2^7 \cdot 3^4 \cdot 5$).

Consider the factorisation $W = (L_{uv}N/N)(L_{vv_1}N/N)$.
Since $N$ is solvable, both factors either are solvable or have the same non‑solvable composition factors. 
From the classification of factorisations of $\mathrm{P\Gamma U}_4(2)$ in~\cite[Tables~1 and~3]{LPS1990} and~\cite[Proposition~4.1]{LX2021+}, we deduce that either both factors contain $\mathrm{PSU}_4(2)$, or both are solvable and appear in Table~\ref{tb:exceptionalfacs}.
We note that, in the latter case only one factor has order divisible by $5$.
Without loss of generality, we may assume that $L_{uv}N/N$ is the factor containing $\mathrm{PSU}_4(2) $ or having  order not divisible by $5$. 
 
In the natural $\mathbb{Q}W$-module $V$, computation in {\sc Magma} shows that the restriction of $V$ to $H_{uv}N/N$ is irreducible.
Moreover, the same holds for the reductions modulo $2$ and modulo $5$ (i.e., for the corresponding modules over $\mathbb{F}_2$ and $\mathbb{F}_5$).
 
\smallskip

\noindent{\bf Claim 1}. Each prime divisor $r$ of $q-1$ satisfies  $r \in \{2,3 \}  $.
Moreover, $f \leq 2 $. 

We first show that no prime $r \notin \{2,3,5\}$ can divide $q-1$.
Assume such an $r$ exists. 
Then $H_v$ contains a unique subgroup $R\cong  \mathrm{C}_{r}^6$.
Since $r$ divides $|L_v|$, the factorisation $L_v = L_{uv}L_{vv_1}$ forces $r$ to divide both $|L_{uv}|$ and $|L_{vv_1}|$.
Thus, $R \cap L_{uv} \neq 1$.
Because $L_{uv}N/N$ acts irreducibly on the $\mathbb{F}_r W$-module $R $, we must have $R \le L_{uv}$.
Since $L_{vv_1}=L_{uv}^h$ and $R$ is unique in $H_v$, we have $R^h = R$, contradicting Lemma~\ref{lm:Hvnormal}.

Next we exclude $r = 5$.
Let $r=5 \mid (q-1)$. Then  $N$ contains a unique subgroup $R\cong \mathrm{C}_{5}^6 $.
View $R$ as an $\mathbb{F}_5 W$-module.
Computation reveals that for an element $x \in W$ of order $5$, the restriction of the $\mathbb{F}_5 W$-module $R$ to $\langle x \rangle$ splits as $\mathbb{F}_5 \oplus \mathbb{F}_5^5$ with both summands irreducible.
This implies that $R$ is also the unique subgroup of $L_v$ isomorphic to $ \mathrm{C}_{5}^6$. 
Again, the irreducibility of $L_{uv}N/N$ on $R$ forces $R \le L_{uv}$.
Then we can obtain  $R^h = R$, contradicting Lemma~\ref{lm:Hvnormal}.

Thus, every prime divisor of $q-1$ belongs to $\{2,3\}$. 
If $f>2$, then  a prime $s\in \mathrm{ppd}(p,f)$ satisfies  $s\mid (q-1)$ and $s\geq f+1>3$, a contradiction. Therefore $f \le 2$.
This completes the proof of the Claim.

\smallskip
By  Claim 1, every prime divisor of $q-1$ is $2$ or $3$, so $N$ is a $\{2,3\}$-group.
If both $L_{uv}N/N$ and $L_{vv_1}N/N$ are solvable, then Table~\ref{tb:exceptionalfacs} shows that one factor has order divisible by $5$ while the other does not.
But $|L_{uv}|$ and $|L_{vv_1}|$ have the same $5$-part, a contradiction.
Therefore, both $L_{uv}N/N$ and $L_{vv_1}N/N$ contain $\mathrm{PSU}_4(2)$.

Suppose that $(q-1)_2=1$. Then $q-1=3^a$ for some positive integer $a $.
By Claim 1, we have $f\leq 2$.
If $f=1$, then \(q = 3^a + 1\) is even, whereas \(q = p\) is an odd prime greater than \(2\); this is impossible.  
If \(f = 2\), we have \(q - 1 = (p+1)(p-1) = 3^a\). If $p$ is odd, then $(p+1,p-1)=2$, contradicting the fact that the right-hand side is a power of \(3\).  
Therefore, $p=2$, forcing $q=4$, which contradicts the assumption $q>4$.
Consequently, the case $(q-1)_2=1$ can not happen.
 
Write $(q-1)_2 = 2^a$ for some positive integer $a $, and let $R$ be the unique subgroup of $N$ isomorphic to $\mathrm{C}_2^6$.
Then $|L_v|_2 = 2^{6a+7}$ (since $|W|_2 = 2^7$), and from $L_v = L_{uv}L_{vv_1}$ we obtain $|L_{uv}|_2 \ge 2^{3a+4}$.

Assume first that $a \ge 2$.
Because $|W|_2 = 2^7$, we have $|L_{uv} \cap N|_2 \ge 2^{3a+4}/2^7 = 2^{3a-3} \ge 2^3$, so in particular $R \cap L_{uv} \neq 1$.
Since $L_{uv}N/N$ (containing $\mathrm{PSU}_4(2)$) acts irreducibly on  $R $, we must have $R \le L_{uv}$.
By symmetry $R \le L_{vv_1}$. 
Since $R$ is normal in $L_{uv}$, we conclude from $L_{vv_1}=L_{uv}^h$ that $R^h$ is also normal in $L_{vv_1}$.
Note that $R$ is the unique normal subgroup of $L_{vv_1}$ that isomorphic to $\mathrm{C}_{2}^6$.
It follows $R^h=R$, contradicting Lemma~\ref{lm:Hvnormal}.

Therefore, $a=1$ and $\vert R\cap L_{uv}\vert_2=1$.
Then $\vert L_{uv}\vert_2=2^7$, and  the valency of $\mathit{\Gamma}$, $\vert L_v\vert/\vert L_{uv}\vert$,  is divisible by $2^6$.
Let $\vert H_v/L_v\vert_2=2^{i}$ for some $i\in \{0,1,2\}$ (note $f\leq 2$).
Then $\vert H_v\vert=2^{13+i}$ and $\vert H_{uv}\vert_2 \leq 2^{7+i} $.
Since $s\geq 3$, by Lemma~\ref{lm:HvHuv}  $ \vert H_v\vert_2^{2}\leq \vert H_{uv}\vert_2^{3}$, yielding $5\leq i$, a contradiction. 
\end{proof}

\begin{lemma}\label{lm:2E6(q+1)6}
Assume Hypothesis~$\ref{hy:1}$ and  $L={}^2E_6(q)$.  If case (xi) holds, i.e.,  $L_v = \mathrm{C}_{q+1}^6/(3,q+1).W(E_6)$ with $q>2$, then $s\leq 2$. 
\end{lemma}  

\begin{proof} 
The argument follows the same pattern as Lemma~\ref{lm:E6(q-1)6}. 
Using reasoning analogous to that in  Claim 1 of Lemma~\ref{lm:E6(q-1)6}, we can prove that every prime divisor of $ q+1 $ is $2$ or $3$, and either $f=1$ or $q=8$.
With argument identical to the $2$-part analysis in Lemma~\ref{lm:E6(q-1)6}, we can rule out the case where $(q+1)_2>1$.
It remains to treat the case $(q+1)_2=1$, that is, $q+1$ is a $3$-power.
In this case we must have $q=8$, because otherwise $f=1$ and so $p=q=3^a-1 $ for some integer $a$, forcing $q=2$, which contradicts the assumption $q>2$. 
Then $|L_v|_3=3^{15}$, and so $|L_{uv}|_3=|L_{vv_1}|_3\geq 3^8$.
Let $R$ be the unique subgroup of $N=\mathrm{C}_{9}^6/3$ isomorphic to $\mathrm{C}_3^6$.
Since $|W|_3=3^4$, we have $L_{uv} \cap R>1$.
Using an argument similar to  the $2$-part analysis, we may also exclude the case  $q=8$. 
\end{proof}

\section*{Acknowledgments} 
We would like to express our sincere gratitude to the anonymous reviewer and the editor for their careful reading of our manuscript and their insightful comments and suggestions, which have greatly improved the quality of this paper. The authors are also grateful to Prof Michael Giudici and Prof Cheryl E Praeger for their precious comments.
This work was supported by the National Natural Science Foundation of China (12301461,12331013) and 
 the Deutsche Forschungs-gemeinschaft (DFG, German Research Foundation)–Project-ID 491392403–TRR 358.

 

 
 

\end{document}